%% file: Hodge_Fermat_Rationals.tex
\documentclass[12pt,leqno]{article}
\usepackage[left=3cm, right=3cm]{geometry}

\usepackage{bbm, dsfont}
\usepackage{graphicx, amsfonts, amsthm, amsxtra, amssymb, verbatim, makeidx,longtable,booktabs}
\usepackage{subeqnarray, relsize}
\usepackage[mathscr]{euscript}
\usepackage{hyperref, tikz-cd}
\usepackage{aliascnt}
\usepackage[nameinlink,capitalize,noabbrev]{cleveref}
\usepackage{mathtools}
\hypersetup{
    colorlinks=true,       
    linkcolor=blue,          
    citecolor=blue,        
    filecolor=blue,      
    urlcolor=blue           
}
\usepackage{packages/rufino}
\usepackage{stmaryrd}
\usepackage{float}

\newtheorem{theo}{Theorem}[section]
\newtheorem{coro}{Corollary}[section]

\newtheorem{prop}{Proposition}[section]
\newtheorem{conj}{Conjecture}[section]

\theoremstyle{remark}
\newtheorem{rem}{Remark}[section]

\theoremstyle{definition}
\newtheorem{defi}{Definition}[section]

\newenvironment{proof*}
  {\begin{proof}}
  {\end{proof}}

\begin{document}

\include{notations}

\def\Norm{\text N}

\begin{center}
{\LARGE\bf Lengths of Hodge characters in Fermat varieties
}
\\
\vspace{.1in} {\large {\sc Maximiliano Miranda, Hossein Movasati, Lucas Rufino and Roberto Villaflor}}

\end{center}
{\it Abstract: We revisit the Hodge conjecture for Fermat (and weighted Fermat) varieties. We review the reduction of the Hodge conjecture in terms of Hodge characters introduced by Shioda and its further reduction in terms of the formal module of tuples introduced by Aoki. Following Aoki, we introduce several lengths on this formal module, and by means of Kang's theorem on the Hodge conjecture for Fermat fourfolds we reduce the conjecture to bound the lengths of exceptional tuples by 6. By applying length reduction algorithms we prove the Hodge conjecture for all Fermat varieties (of any dimension) of degree less than 65 and different from 44, 51 and 52. The main novelty of our length reduction method is the introduction of a lift operation on the degree of the tuple. 
}


\section{Introduction}
\input{sections/s01}

\section{Hodge Cycles in Weighted Fermat Varieties}\label{sec:2}
\input{sections/s02}

\section{Operations on Characters}\label{sec:3}
\input{sections/s03}

\section{Formal Module of Tuples}\label{sec:4}
\input{sections/s04}




\bibliography{biblio.bib}

\bibliographystyle{alpha}

\bigskip

\noindent{\sc Departamento de Matemática, Universidad Técnica Federico Santa María}\\
{\sc Avenida España 1680, Valparaíso, Chile}\\
\textit{Email address:} {\tt maximiliano.mirandah@usm.cl}
\\
\\
\noindent{\sc Instituto de Matem\'atica Pura e Aplicada, IMPA}\\
{\sc Estrada Dona Castorina, 110, 22460-320, Rio de Janeiro, RJ, Brazil}\\
\textit{Email address:} {\tt hossein@impa.br}
\\
\\
\noindent{\sc Graduate School of Mathematics, Nagoya University}\\
{\sc Chikusa-ku, Nagoya, 464-8602, Japan}\\
\textit{Email address:} {\tt lucas.martelotte@impa.br}
\\
\\
\noindent{\sc Departamento de Matemática, Universidad Técnica Federico Santa María}\\
{\sc Avenida España 1680, Valparaíso, Chile}\\
\textit{Email address:} {\tt roberto.villaflor@usm.cl}

\end{document}

%% file: notations.tex
\def\Nn{{\sf n}}
\def\Am{{\sf A}}

\def\Z{\mathbb{Z}}                   
\def\Q{\mathbb{Q}}                   
\def\C{\mathbb{C}}                   
\def\N{\mathbb{N}}                   
\def\Ff{\mathbb{F}}                  
\def\uhp{{\mathbb H}}                
\def\A{\mathbb{A}}                   
\def\dR{{\rm dR}}                    
\def\F{{\cal F}}                     
\def\Sp{{\rm Sp}}                    
\def\Gm{\mathbb{G}_m}                 
\def\Ga{\mathbb{G}_a}                 
\def\Tr{{\rm Tr}}                      
\def\tr{{{\mathsf t}{\mathsf r}}}                 
\def\spec{{\rm Spec}}            
\def\proj{{\rm Proj}}
\def\ker{{\rm ker}}              
\def\GL{{\rm GL}}                

\def\k{{\sf k}}                     
\def\ring{{\sf R}}                   
\def\sk{{\mathfrak k }}             
\def\sring{{\mathfrak R }}          

\def\X{{\sf X}}                      
\def\T{{\sf T}}                      
\def\V{{    V}}                   

\def\Ts{{\sf S}}
\def\cmv{{\sf M}}                    
\def\BG{{\sf G}}                       
\def\podu{{\sf pd}}                   
\def\ped{{\sf U}}                    
\def\per{{\sf  P}}                   
\def\gm{{\sf  A}}                    
\def\gma{{\sf  B}}                   
\def\ben{{\sf b}}                    

\def\Rav{{\mathfrak M }}                     
\def\Ram{{\cal C}}                         
\def\Rap{{i(\Lie(\BG))}}                    

\def\nov{{  n}}                    
\def\mov{{  m}}                    
\def\Yuk{{\sf Y}}                     
\def\Ra{{\sf R}}                      

\def\Da{{\sf D}}                      

\def\hn{{\sf h}}                      
\def\cpe{{\sf C}}                     
\def\g{{\sf g}}                       
\def\t{{   t}}                       
\def\v{{   v}}                       

\def\pedo{{\sf  \Pi}}                  

\def\Der{{\rm Der}}                   
\def\MMF{{\sf MF}}                    
\def\codim{{\rm codim}}                
\def\dim{{\rm    dim}}                
\def\Lie{{\rm Lie}}                   

\def\u{{\sf u}}                       

\def\imh{{  \Psi}}                 
\def\imc{{  \Phi }}                  
\def\stab{{\rm Stab }}               
\def\Vec{{\Theta}}                 
\def\prim{{0}}                  
\def\Zero{{\rm Zero}}                  

\def\Fg{{\sf F}}     
\def\hol{{\rm hol}}  
\def\non{{\rm non}}  
\def\alg{{\rm alg}}  
\def\an{{\rm an}}   
\def\for{{\rm for}}  

\def\bcov{{\rm \O_\T}}       

\def\leaves{{\cal L}}        

\def\Hse{{\rm HS}}        
\def\Hpo{{\rm HP}}        
\def\Hfu{{\rm HF}}        
\def\Hsc{{\rm Hilb}}     

\def\TS{\mathlarger{{\bf T}}}                
\def\IS{\mathlarger{{\cal I}}}                

\def\vf{{\sf v}}                      
\def\wf{{\sf w}}                      

\def\red{{\rm red}}                           

\def\Ua{{   L}}                      
\def\plc{{ Z_\infty}}    

\def\gru{\mu} 
\def\pg{{ \sf S}}               
\def\group{{ G}}            

\def\GM{{\rm GM}}

\def\perr{{\sf q}}        
\def\perdo{{\cal K}}   
\def\sfl{{\mathrm F}} 
\def\sp{{\mathbb S}}  

\newcommand\diff[1]{\frac{d #1}{dz}} 
\def\End{{\rm End}}              

\def\sing{{\rm Sing}}            
\def\cha{{\rm char}}             
\def\Gal{{\rm Gal}}              
\def\jacob{{\rm jacob}}          
\def\tjurina{{\rm tjurina}}      
\newcommand\Pn[1]{\mathbb{P}^{#1}}   
\def\P{\mathbb{P}}                   
\def\Ff{\mathbb{F}}                  

\def\O{{\cal O}}                     
\def\as{\mathbb{U}}                  
\def\ring{{\mathsf R}}                         
\def\R{\mathbb{R}}                   

\newcommand\ep[1]{e^{\frac{2\pi i}{#1}}}
\newcommand\HH[2]{H^{#2}(#1)}        
\def\Mat{{\rm Mat}}              
\newcommand{\mat}[4]{
     \begin{bmatrix}
            #1 & #2 \\
            #3 & #4
       \end{bmatrix}
    }                                
\newcommand{\matt}[2]{
     \begin{bmatrix}                 
            #1   \\
            #2
       \end{bmatrix}
    }
\def\cl{{\rm cl}}                

\def\hc{{\mathsf H}}                 
\def\Hb{{\cal H}}                    
\def\pese{{\sf P}}                  

\def\PP{\tilde{\cal P}}              
\def\K{{\mathbb K}}                  

\def\M{{\cal M}}
\def\RR{{\cal R}}
\newcommand\Hi[1]{\mathbb{P}^{#1}_\infty}
\def\pt{\mathbb{C}[t]}               
\def\W{{\cal W}}                     
\def\gr{{\rm Gr}}                
\def\Im{{\rm Im}}                
\def\Re{{\rm Re}}                
\def\depth{{\rm depth}}
\newcommand\SL[2]{{\rm SL}(#1, #2)}    
\def\sl{{\rm SL}}                    
\newcommand\PSL[2]{{\rm PSL}(#1, #2)}  
\def\Resi{{\rm Resi}}              

\def\L{{\cal L}}                     
\def\Aut{{\rm Aut}}              
\def\any{R}                          
\newcommand\ovl[1]{\overline{#1}}    

\newcommand\mf[2]{{M}^{#1}_{#2}}     
\newcommand\mfn[2]{{\tilde M}^{#1}_{#2}}     

\newcommand\bn[2]{\binom{#1}{#2}}    
\def\ja{{\rm j}}                 
\def\Sc{\mathsf{S}}                  
\newcommand\es[1]{g_{#1}}            
\newcommand\WW{{\mathsf W}}          
\newcommand\Ss{{\cal O}}             
\def\rank{{\rm rank}}                
\def\Dif{{\cal D}}                   
\def\gcd{{\rm gcd}}                  
\def\zedi{{\rm ZD}}                  
\def\BM{{\mathsf H}}                 
\def\plf{{\sf pl}}                             
\def\sgn{{\rm sgn}}                      
\def\diag{{\rm diag}}                   
\def\hodge{{\rm Hdg}}
\def\HF{{\sf F}}                                
\def\WF{{\sf W}}                               
\def\HV{{\sf HV}}                                
\def\pol{{\rm pole}}                               
\def\bafi{{\sf r}}
\def\Id{{\rm Id}}                               
\def\gms{{\sf M}}                           
\def\Iso{{\rm Iso}}                           

\def\hl{{\rm L}}    
\def\imF{{\rm F}}
\def\imG{{\rm G}}

\def\HL{{\rm Ho}}     
\def\NLL{{\rm NL}}   

\def\RG{{\bf G}}          
\def\rg{{\bf g}}     
\def\rbullet{{\cdot}}
\def\Ld{{\cal L}}      
\def\Ro{{\rm R}}     
\def\ZS{{\rm ZeSc}}     
\def\ZI{{\rm ZeId}}     
 \def\integ{{\rm Int}}  

\def\tmap{{\sf t}}

\def\ivhs{{\rm IVHS}}    
\def\ivhsmaps{{{\Delta}_{}}}   
\def\sch{{\rm Sch}}   
\def\mk{{\mathfrak  m}}   
\def\pk{{\mathfrak  p}}   
\def\qk{{\mathfrak  q}}   

\newcommand\licy[1]{{\mathbb P}^{#1}} 

\def\SS{\mathscr{S}}    

%% file: sections/s01.tex
The purpose of this article is twofold. On the one hand, we compile the known results in the literature on the Hodge conjecture for Fermat varieties, and we explore to what extent the conjecture can be proved by means of the explicit
families of algebraic cycles that are available on them, namely the linear cycles of Ran \cite{Ran1980} and Shioda \cite{sh79}, the cycles introduced by Aoki \cite{aoki1987}, and everything that can be built out of them by the geometric operations of join, deletion, permutation and pull-back. On the other hand, we introduce the appropriate language to work with the Hodge conjecture in weighted Fermat varieties. This is done in order to use it in an upcoming article, where we propose a method to search for new algebraic cycles in Fermat varieties, based on searches performed in weighted Fermat varieties. Our main result is the following.
 
\begin{theo}[\cref{mainthm}]\label{thm1}
The Hodge conjecture holds for the Fermat variety $X^n_d$ of any dimension $n$ and any degree $d<65$ with $d\neq 44,51,52$. The same holds for every weighted Fermat variety $X^n_m$ whose degree $d=\operatorname{lcm}(m)$ satisfies these conditions.
\end{theo}
 
Prior to this, the conjecture was known (in arbitrary dimension) for $d\leq 20$ by the work of Shioda \cite{sh79}, for $d$ a prime power or twice a prime power by the work of Aoki \cite{aoki1987}, and for $d=21$ by the computer assisted verification of da Silva Jr. \cite{da2021notes}. In particular, \cref{thm1} settles the degrees $28$, $33$, $35$, $39$, $55$, $56$ and $57$, and it leaves $44$, $51$ and $52$ as the only degrees below $65$ for which the conjecture remains open. This follows from the general fact, proved by Aoki \cite{Aoki1983}, that in order to prove the Hodge conjecture for all degree $d$ Fermat varieties, it is enough to prove it for Fermat varieties of degree $q\in Q$ dividing $d$, where $Q$ denotes the set of products of an even number of distinct elements of $P=\{4\}\cup\{p\geq 3\ \text{prime}\}$ (see \cref{correddivQ}). In particular, to settle all the degree $d< 3\cdot 4\cdot 5\cdot 7=420$ cases, the relevant open cases are of degree $pq$ for $p\neq q$ and $p,q\in P$. In particular, the degrees of the form $2^\alpha\cdot 3^\beta\cdot 5^\gamma$ were already known by \cite{sh79,Aoki1983} and our result implies the cases of the form $2^\alpha\cdot 3^\beta\cdot 7^\gamma$, $2^\alpha\cdot 5^\beta\cdot 7^\gamma$, $3^\alpha\cdot 5^\beta\cdot 11^\gamma$ and $2\cdot 3^\alpha\cdot 5^\beta\cdot 11^\gamma$ for all $\alpha,\beta,\gamma\in\N$.
 
In order to obtain this result we work with the formal module of tuples $R$, introduced by Aoki \cite{Aoki1983}. It is the free $\mathbb{Z}$-module generated by $(\mathbb{Q}/\mathbb{Z})\setminus\{0\}$, inside which a character becomes a sum of formal symbols, the join of two characters becomes their sum, the deletion becomes a subtraction, a permutation becomes the identity, and the lift of a character to a Fermat variety of larger degree becomes an identification. Hodge, algebraic, linear and standard characters generate submodules (see \cref{sec:4.1})
\[
D+\widetilde{S}\;\subseteq\; C\;\subseteq\; B\;\subseteq\; R ,
\]
and, since the four operations above preserve algebraicity, the Hodge conjecture for all weighted Fermat varieties of degree $d$ follows from the equality $C_d=B_d$ (\cref{thm:reduce_hodge_conjecture_to_Bm}). By a theorem of Aoki (\cref{thm:generators_of_Bm}), the quotient $B_d/(D_d+S_d)$ is generated by finitely many \emph{exceptional tuples} $\xi_q$, for each $q\in Q$ divisor of $d$. Whether these exceptional tuples are algebraic is precisely what is left of the conjecture. In this module we introduce several notions of length. The naive length $\ell(\alpha)$ counts the number of entries of a tuple, and by taking quotients by the submodules $D$, $D+\widetilde{S}$ and $C$ we obtain lengths $\ell_1$, $\ell_2$ and $\mathscr{L}$, which measure how far a tuple is from being linear, standard or algebraic. Their geometric meaning is the following: a tuple of length $\ell$ lives on a Fermat variety of dimension $\ell-2$ for $\ell$ even, and of dimension $\ell-1$ for $\ell$ odd, so reducing the length of a tuple amounts to descending to a Fermat variety of smaller dimension. Since the Hodge conjecture is known for Fermat fourfolds (by a theorem of Kang \cite{kang2016refined}), this reduces the Hodge conjecture for Fermat varieties to bound by $6$ the lengths of the exceptional tuples $\xi_q$ (\cref{prop:lengthredto6}).
 
We present two length reduction algorithms to tackle this problem. The first one, that we call the Aoki-Shioda algorithm, reduces a tuple by the generators of $D_d+S_d$ inside the module $R_d$ of a fixed degree $d$, it amounts to an integral linear programming problem and it provides upper bounds for $\ell_{2,q}(\xi_q)$. The second one, that we call the lifted Aoki-Shioda algorithm, improves it by incorporating the lift operation, which reduces the
length of a tuple at the cost of increasing the degree of the ambient Fermat variety, and provides upper bounds for $\ell_2(\xi_q)$. This is what happens, for instance, with the exceptional tuple of degree $35$, whose reduction is supported on a weighted Fermat fourfold of degree $70$ (see \cref{tab:lengths_of_representatives_for_exceptional_cycles_AS} and \cref{tab:lengths_of_representatives_for_exceptional_cycles_liftedAS}). The bounds produced by both algorithms are collected in the tables of \cref{sec:4.3}, and \cref{thm1} is a direct consequence of them.
 
In order to explain the lift operation, and to explain the geometric origin of all the operations that we perform on Hodge characters, we decided to give a self-contained exposition of all the elementary results we use, citing without
proof only the deeper ones. This is why we expect that the text can also serve as a quick introduction to the Hodge conjecture on Fermat varieties.
 
Let us finally mention what our tables leave open. The Fermat variety of smallest degree and dimension for which the Hodge conjecture remains open is $X^6_{44}$, and our computations suggest that the conjecture for Fermat sixfolds of any degree is equivalent to the algebraicity of the four exceptional tuples $\xi_{44}$, $\xi_{51}$, $\xi_{52}$ and $\xi_{68}$. More generally, our bounds suggest that the lengths of the exceptional tuples converge to infinity with the degree (\cref{conj1}), which would reduce the Hodge conjecture in each fixed dimension to a finite number of degrees.
 
The article is organized as follows. In \cref{sec:2} we describe the space of Hodge cycles of a weighted Fermat variety in terms of characters. After recalling the Hodge structure of a projective orbifold and the Griffiths basis of a quasi-smooth hypersurface of a weighted projective space, we compute the periods of the spectral basis in terms of the Beta function and we determine the Galois action on it, which yields the weighted version of Shioda's description of Hodge cycles. In \cref{sec:3} we introduce the four operations on characters: lift and permutation, coming from pull-back along finite maps, and join and deletion, coming from the inductive structure of Fermat varieties. We prove that all of them preserve both the Hodge and the algebraic property, and we recall the two known families of algebraic characters, the linear ones of Ran and Shioda and the standard ones of Aoki. \cref{sec:4} is devoted to the formal module of tuples. We recall Aoki's generation theorem for the exceptional tuples, we introduce the several lengths, we prove the reduction of the Hodge conjecture to a bound by $6$, and we describe the two length reduction algorithms together with the tables they produce, from which our main theorem follows.
 
\subsection*{Acknowledgements}
The fourth author was supported by Fondecyt ANID regular grant 1240101 and Fondecyt ANID initiation grant 11251404.

%% file: sections/s02.tex
In this section we describe the space of Hodge cycles in weighted Fermat varieties in terms of the so called Griffiths basis and the weighted version of Shioda's description of Hodge cycles in terms of characters \cite{sh79}. Let us begin by introducing these varieties.

Let $n\in \N$. Given a vector $m=(m_0,\ldots,m_{n+1})\in \N^{n+2}$, we say that the weighted Fermat variety of weight $m$ is the degree $d:={\rm lcm}(m)$ quasi-smooth hypersurface
$$
X^n_m:=\{x\in\P^v: x_0^{m_0}+\cdots+x_{n+1}^{m_{n+1}}=0\}
$$
of the weighted projective space $\P^v$ with weight vector $v=(v_0,\ldots,v_{n+1})$ given by $v_i=\frac{d}{m_i}$ for every $i=0,\ldots,n+1$. In the classical homogeneous case, where $m_0=\cdots=m_{n+1}=d$, we simply denote it by
$$
X^n_d:=\{x\in\P^{n+1}:x_0^d+\cdots+x_{n+1}^d=0\}.
$$
In spite that in general these varieties ($X^n_m$ and $\P^v$) might be singular, they are always projective orbifolds (or $V$-manifolds) and so they have a natural pure Hodge structure in their rational cohomology groups (c.f. \cite[Section 2.5]{SP2008}). We briefly recall this structure in the following subsection.

\subsection{Hodge cycles in projective orbifolds}

Let $X$ be an orbifold. Let $i:X^{sm}\hookrightarrow X$ be the inclusion of the smooth locus of $X$, and let $\widetilde{\Omega}_X^p:=i_*\Omega_{X^{sm}}^p$. The complex $(\widetilde{\Omega}_X^\bullet,d)$ is a resolution of the constant sheaf $\C$ on $X$ and so we have the spectral sequence associated to the naive filtration
$$
E^{p,q}_1=H^q(X,\widetilde{\Omega}_X^p)\Rightarrow \mathbb{H}^{k}(X,\widetilde{\Omega}_X^\bullet)=H^{k}(X,\C) .
$$
For $X$ a  projective variety, the spectral sequence degenerates at $E_1$ (c.f. \cite{st77}) and moreover determines a Hodge structure on $H^k(X,\C)$. Another way to describe this Hodge structure is by considering any resolution of singularities $\phi:M\rightarrow X$ and taking the Hodge structure on $H^k(X,\Q)$ which makes $\phi^*:H^k(X,\Q)\hookrightarrow H^k(M,\Q)$ a morphism of Hodge structures, i.e. 
$$
H^{p,q}(X):=H^{k}(X,\C)\cap (\phi^*)^{-1}(H^{p,q}(M))
$$
for $p+q=k$. This structure is independent of the resolution and satisfies
$$
H^{p,q}(X)\simeq H^q(X,\widetilde{\Omega}_X^p) .
$$

\begin{defi}
Let $X$ be a projective orbifold. A rational class $\eta\in H^{2p}(X,\Q)$ is called a Hodge cycle if $\eta\in H^{p,p}(X)$. This is equivalent to say that $\eta\in F^pH^{2p}(X,\C)= \mathbb{H}^{2p}(X,\widetilde{\Omega}_X^{\bullet\ge p})$.
\end{defi}

It is natural to extend the Hodge conjecture to this context as follows. Given any codimension $p$ algebraic subvariety $Z\subseteq X$, it defines (by triangulation) a homology cycle $\delta_Z\in H_{2n-2p}(X,\Q)$. Since $X$ is rationally smooth, it satisfies the Poincar\'e duality with $\Q$-coefficients (see \cite[Definition 11.4.3, Example 11.4.4, Section 12.4]{cox2011toric})
$$
H_{2n-2p}(X,\Q)\simeq H^{2n-2p}(X,\Q)^\vee\simeq H^{2p}(X,\Q) .
$$
Hence, we get a class $[Z]\in H^{2p}(X,\Q)$. 

\begin{defi}
A rational class $\eta\in H^{2p}(X,\Q)$ is called an algebraic cycle if it is of the form $\eta=\sum_{i=1}^kq_i[Z_i]$ for some $q_i\in\Q$ and a collection $Z_i\subseteq X$ of codimension $p$ algebraic subvarieties. We denote the space of algebraic cycles by $H^{2p}(X,\Q)_\text{alg}$.   
\end{defi}

It is easy to see that in this context every algebraic cycle is a Hodge cycle. In fact, given $Z\subseteq X$ a codimension $p$ algebraic subvariety, let us denote by $W\subseteq M$ its strict transform under $\phi$. Then $\phi_*\delta_W=\delta_Z$ and so for every class $\mu\in H^{r,s}(X)$ with $r+s=2n-2p$ and $r\neq s$
$$
\text{Tr}([Z]\cup \mu)=\mu(\delta_Z)=\mu(\phi_*\delta_W)=\phi^*\mu(\delta_W)=\text{Tr}([W]\cup\phi^*\mu)=0 \,
$$
thus $[Z]\in H^{p,p}(X)$ is a Hodge cycle. In this way we extend the Hodge conjecture to $X$ by asking whether every Hodge cycle of $X$ is an algebraic cycle.

\begin{rem}
Let $f:X\rightarrow Y$ be a finite map between projective orbifolds. It induces in homology a direct image map
$$
f_*:H_{2n-2p}(X,\Z)\to H_{2n-2p}(Y,\Z) ,
$$
while the pre-image induces a pull-back map
$$
f^*:H_{2n-2p}(Y,\Z)\to H_{2n-2p}(X,\Z) ,
$$
such that $f_*f^*=\deg(f)\cdot Id_{H_{2n-2p}(Y,\Z)}$. Considering these maps with $\Q$ coefficients, dualizing, and using the Poincar\'e duality we obtain a pull-back and a push-forward map in cohomology
$$
f^*:H^{2p}(Y,\Q)\to H^{2p}(X,\Q) ,
$$
$$
f_*:H^{2p}(X,\Q)\to H^{2p}(Y,\Q) ,
$$
such that $f_*f^*=\deg(f)\cdot Id_{H^{2p}(Y,\Q)}$. In particular, $f^*$ is injective, and so we can identify $$H^{2p}(Y,\Q)\subseteq H^{2p}(X,\Q)$$ as a sub-Hodge structure via $f^*$. Hence
$$
H^{2p}(Y,\Q)\cap H^{p,p}(Y)=H^{2p}(X,\Q)\cap H^{p,p}(X)\cap H^{2p}(Y,\C) .
$$
\end{rem}

\begin{prop}\label{propalgcycfinitemap}
Let $f:X\to Y$ be a finite map between projective orbifolds. For every algebraic cycle $W\in \text{CH}^p(X)$ and $Z\in \text{CH}^p(Y)$
$$
f^*[Z]=\deg(f)\cdot [f^{-1}(Z)] \ \ \text{ and } \ \ f_*[W]=\frac{1}{\deg(f)}[f(W)] .
$$
In consequence, 
$$
H^{2p}(Y,\Q)_\text{alg}=H^{2p}(X,\Q)_\text{alg}\cap H^{2p}(Y,\C) .
$$
In particular, if $X$ satisfies the Hodge conjecture, then $Y$ also does.
\end{prop}

\begin{proof}
For the first equality, let $\omega\in H^{2n-2p}(X,\Q)$, then
\begin{eqnarray*}
\text{Tr}(f^*[Z]\cup \omega) &= &\text{Tr}(f_*f^*[Z]\cup f_*\omega)= \deg(f)\cdot\text{Tr}([Z]\cup f_*\omega)
\\
&=& {\deg(f)}\cdot f_*\omega(\eta_Z) {\deg(f)}\cdot\omega(f^*\eta_{Z})\\
&=& {\deg(f)}\cdot\omega(\eta_{f^{-1}(Z)})={\deg(f)}\cdot\text{Tr}([f^{-1}(Z)]\cup\omega) .
\end{eqnarray*}
For the second let $\mu\in H^{2n-2p}(Y,\Q)$, then 
\begin{eqnarray*}
\text{Tr}(f_*[W]\cup \mu) &=& \frac{1}{\deg(f)}\text{Tr}(f_*[W]\cup f_*f^*\mu)=\frac{1}{\deg(f)}\text{Tr}([W]\cup f^*\mu)
\\
&=& \frac{1}{\deg(f)}f^*\mu(\eta_W)=\frac{1}{\deg(f)}\mu(f_*\eta_W)\\
&=& \frac{1}{\deg(f)}\mu(\eta_{f(W)})=\frac{1}{\deg(f)}\text{Tr}([f(W)]\cup\mu) .
\end{eqnarray*}
\end{proof}

An immediate consequence of the above result is the following.

\begin{coro}
If the Hodge conjecture holds for the weighted Fermat variety $X^n_m$, then it also holds for all weighted Fermat varieties $X^n_{m'}$ with $m'\mid m$ (by this we mean $m_i'\mid m_i$ for all $i=0,\ldots,n+1$).
\end{coro}

\begin{proof}
Take the finite map $f:X^n_m\to X^n_{m'}$ given by $$f(x_0:\cdots:x_{n+1})=[x_0^\frac{m_0}{m_0'}:\cdots:x_{n+1}^\frac{m_{n+1}}{m_{n+1}'}] .$$
\end{proof}

\begin{rem}
In particular, if $\text{lcm}(m)=d$, then the Hodge conjecture for $X^n_d$ implies the Hodge conjecture for $X^n_m$.
\end{rem}

\subsection{Hodge structure of quasi-smooth weighted hypersurfaces}

Suppose now that $X=\{F=0\}\subseteq\P^v$ is a quasi-smooth degree $d$ hypersurface, given by $F\in S(\P^v)=\C[x_0,\ldots,x_{n+1}]$ where $\deg(x_i)=v_i$ and $\deg(F)=d$. It is a well-known fact that the Gysin map
$$
H^k(X,\Q)\xrightarrow{i_!} H^{k+2}(\P^v,\Q)
$$
is an isomorphism for $k>\dim(X)$ (see for instance \cite[Proposition 4.1]{Villaflor2024toric}), and that 
$$
H^k(\P^v,\Q)=\begin{cases}
    \Q\cdot [Z] & \text{if $k$ is even and $Z\subseteq\P^v$ is any codimension $\frac{k}{2}$ subvariety}, \\
    0 & \text{otherwise}.
\end{cases}
$$
Furthermore, it is known that $X$ satisfies the Hard Lefschetz theorem \cite[Theorem 12.5.8]{cox2011toric} and so, the only non-trivial Hodge structure of $X$ lies in its middle cohomology $H^n(X,\Q)$.

\begin{defi}
We define the middle primitive cohomology group of $X$ by
$$
H^n(X,\Q)_\prim:=\ker(i_!:H^n(X,\Q)\rightarrow H^{n+2}(\P^v,\Q)).
$$
We also define $H^n(X,K)_\prim:=H^n(X,\Q)_0\otimes_\Q K$ for any field extension $K/\Q$ and $H^{p,q}(X)_\prim:=H^{p,q}(X)\cap H^n(X,\C)_\prim$ for $p+q=n$.
\end{defi}

For quasi-smooth weighted hypersurfaces, the middle cohomology splits as follows
$$
H^n(X,\Q)=H^n(X,\Q)_\prim\oplus H^n(\P^v,\Q) .
$$
This splitting follows easily from the following diagram
\begin{center}
\begin{tikzpicture}[xscale=1.5,yscale=1]

\path       
      node   (m22) at (0,2) {$H^{n}(X,\Q)$} 
      node   (m23) at (2,2) {$H^{n+2}(\P^v,\Q)$} 
      node   (m33) at (1,0.5) {$H^{n}(\P^v,\Q)$};
      { 
      \draw[->]   (m33)  edge node[left] {$i^*$} (m22);
      \draw[->]   (m33)  edge node[right] {$L$} (m23);
      \draw[->]   (m22)  edge node[above] {$i_!$} (m23);
      
       }
\end{tikzpicture}
\end{center}
where $L$ is the Hard Lefschetz isomorphism. Hence, one can reduce the Hodge conjecture to the primitive Hodge cycles of hypersurfaces of even dimension $n$. The following result, due to Steenbrink \cite{st77} and to Batyrev-Cox in the toric setting \cite[Theorem 10.13]{batyrev1994hodge}, is the natural generalization of the so called Griffiths basis \cite[Corollary 6.12]{vo03} for quasi-smooth hypersurfaces of $\P^v$. It tells us how to compute each piece of the Hodge structure of $H^n(X,\C)_\prim$ in terms of the Jacobian ring of $X$, given by $$
R^F:=\C[x_0,\ldots,x_{n+1}]/J^F,\ \hbox{ where }
J^F=( \frac{\partial F}{\partial x_0},\ldots,\frac{\partial F}{\partial x_{n+1}})
$$
is the Jacobian ideal.

\begin{theo}[Steenbrink, Batyrev-Cox]
\label{thmbatycox} Let $X=\{F=0\}\subseteq \P^v$ be a quasi-smooth hypersurface with $\deg(F)=d$. Then for $p\neq \frac{n-1}{2}$ we have an isomorphism
$$
R^F_{d(q+1)-\sum_{i=0}^{n+1}v_i}\simeq H^{p,q}(X)_\prim
$$
$$
P\mapsto \text{res}\left(\frac{P\Omega}{F^{q+1}}\right)^{p,q}
$$
where $p+q=n$ and $\Omega\in H^0(\P^v,\widetilde{\Omega}_{\P^v}^{n+1}(\sum_{i=0}^{n+1}v_i))$ is the canonical generator given by $\Omega=\sum_{i=0}^{n+1}(-1)^iv_ix_idx_0\wedge\cdots\widehat{dx_i}\cdots\wedge dx_{n+1}$.
\end{theo}

\begin{rem}
The residue map can be described by means of the following quasi-isomorphism. Let $X$ be a projective orbifold and $i:Y\hookrightarrow X$ be an ample quasi-smooth divisor. Denote by $U:=X\setminus Y$ and by $j:U\hookrightarrow X$ the inclusion. Then
$$
\widetilde{\Omega}_X^\bullet(\log Y)\hookrightarrow j_*\widetilde{\Omega}_U^\bullet=\widetilde{\Omega}_X(*Y)
$$
is a quasi-isomorphism, and so 
$$
\mathbb{H}^k(X,\widetilde{\Omega}_X^\bullet(\log Y))\simeq \mathbb{H}^k(U,\widetilde{\Omega}_U^\bullet)\simeq H^k(U,\C) .
$$
Using the above isomorphisms, the residue map $\text{res}: H^{k}(U,\C)\rightarrow H^{k-1}(Y,\C)$ is induced in hypercohomology by the usual Poincar\'e residue sequence
\begin{equation}
\label{eqPoincareResSeq}
0\rightarrow \widetilde{\Omega}_X^\bullet\rightarrow \widetilde{\Omega}_X^\bullet(\log Y)\xrightarrow{\text{res}}i_*\widetilde{\Omega}_Y^{\bullet-1}\rightarrow 0 ,
\end{equation}
hence it fits in the long exact sequence
\begin{equation}
\label{eqLerayTGC}
\cdots\rightarrow H^k(X,\C)\xrightarrow{j^*} H^k(U,\C)\xrightarrow{\text{res}}H^{k-1}(Y,\C)\xrightarrow{i_!} H^{k+1}(X,\C)\rightarrow\cdots  .
\end{equation}
In the context of the previous theorem, i.e. for $X=\{F=0\}\subseteq\P^v$, we let $U:=\P^v\setminus X$ and the above exact sequence shows that the residue map is an isomorphism to the primitive cohomology
\begin{equation}
\label{eqResisoC}
\text{res}:H^{n+1}(U,\C)\xrightarrow{\sim}H^n(X,\C)_\prim  .\end{equation}
Moreover, since $X$ is ample, the complex $\widetilde{\Omega}_{\P^v}(*X)$ is acyclic and so
$$
H^{n+1}(U,\C)\simeq \mathbb{H}^{n+1}(\P^v,\widetilde{\Omega}_{\P^v}(*X))\simeq H^{n+1}(\Gamma(\widetilde{\Omega}_{\P^v}^\bullet(*X)),d) .
$$
This explains why the terms of $H^{n+1}(U,\C)$ can be represented by $(n+1)$-forms with poles along $X=\{F=0\}$. 
\end{rem}

\begin{rem}\label{remLTGZ}
The residue isomorphism \eqref{eqResisoC} is dual to the tube mapping $\tau_\varepsilon:H_n(X,\Z)\to H_{n+1}(U,\Z)$ in singular homology, see \cite[Section 4.6]{ho13}. Furthermore, the relation
$$
\int_{\tau_\varepsilon(\delta)}\omega=2\pi i\int_\delta\text{res}(\omega)
$$
implies that the isomorphism
\begin{equation}
2\pi i\cdot \text{res}:H^{n+1}(U,K)\xrightarrow{\sim}H^n(X,K)_\prim 
\end{equation}
holds over any field extension $K/\Q$. Using the notation $K(p):=(2\pi i)^p\cdot K$, the above is equivalent to 
\begin{equation}
\label{eqResisoK}
\text{res}:H^{n+1}(U,K(p))\xrightarrow{\sim}H^n(X,K(p-1))_\prim. 
\end{equation}
\end{rem}

\subsection{Spectral decomposition}

Now that we have established the general framework to work with Hodge structures of weighted quasi-smooth hypersurfaces, let us return to weighted Fermat varieties. In this section we study their middle primitive cohomology in terms of characters, following Shioda's seminal work \cite{sh79}.

Let $n\in\N$, $m=(m_0,\ldots,m_{n+1})\in \N^{n+2}$, $d={\rm lcm}(m)$ and $v_i=d/m_i$ for $i=0,\ldots,n+1$. For any $k\in\N$ us denote by  $\mu_k:=\{ \zeta_{k}^i,\ | \ i\in\Z\},\ \ \zeta_k:=e^\frac{2\pi i}{k}$ the group of $k$-the roots of unity. The group $\mu_{m_0}\times\cdots\times \mu_{m_{n+1}}$ acts on $X^n_m$ diagonally by
$$
(\zeta_{m_0}^{i_0},\ldots,\zeta_{m_{n+1}}^{i_{n+1}})\cdot x:=(\zeta_{m_0}^{i_0}x_0:\cdots:\zeta_{m_{n+1}}^{i_{n+1}}x_{n+1})
$$
with kernel $\lambda\in\mu_d\hookrightarrow(\lambda^{v_0},\ldots,\lambda^{v_{n+1}})\in \mu_{m_0}\times\cdots\times\mu_{m_{n+1}}$. Hence we have a faithful action of the abelian group
$$
G^n_m:=(\mu_{m_0}\times\cdots\times\mu_{m_{n+1}})/\mu_d\curvearrowright X^n_m .
$$
Via pull-back we obtain a natural representation of $G^n_m$ in the middle primitive cohomology $H^n(X^n_m,\Q)_\prim$ which in turn induces a spectral decomposition of this space in terms of the characters 
$$
\widehat{G}^n_m=\hom(G^n_m,\C^\times)\cong\left\{\alpha=(a_0,\ldots,a_{n+1})\in  \prod_{i=0}^{n+1}(\Z/m_i\Z): \sum_{i=0}^{n+1}v_ia_i=0\in \Z/d\Z\right\}
$$
as follows
$$
H^n(X^n_m,\C)_\prim=\bigoplus_{\alpha\in\widehat{G}^n_m}V(\alpha)  ,
$$
where for each $\alpha\in \widehat{G}^n_m$
$$
V(\alpha)=\{\omega\in H^n(X^n_m,\C)_\prim: g^*\omega=\alpha(g)\cdot \omega \ , \ \forall g\in G^n_m\} .
$$
Since $\alpha(g)\in\Q(\zeta_d)$ for all $g\in G^n_m$, it follows that the spectral decomposition is in fact defined over $\Q(\zeta_d)$, hence 
$$
H^n(X^n_m,\Q(\zeta_d))_\prim=\bigoplus_{\alpha\in\widehat{G}^n_m}V(\alpha)_{\Q(\zeta_d)}
$$
for $V(\alpha)_{\Q(\zeta_d)}:=V(\alpha)\cap H^n(X^n_m,\Q(\zeta_d))_\prim$. In order to state the main result of this section let us denote for every $a\in\Z/\ell\Z$ its unique residual representative by
$$
\langle a\rangle\in\{0,1,\ldots,\ell-1\} ,
$$
and for a tuple $\alpha=(a_0,\ldots,a_{n+1})\in \widehat{G}^n_m$, let us set
$$
|\alpha|:=\sum_{i=0}^{n+1}v_i\langle a_i\rangle\in \Z_{\ge 0} .
$$
Consider the following set of characters
\begin{eqnarray}\label{25072026tortuga}
\mathfrak{A}^n_m &:=& \left\{\alpha=(a_0,\ldots,a_{n+1})\in\widehat{G}^n_m\ | \  \forall i,\ a_i\neq 0\right\}\\ \nonumber
&=&
\left\{\alpha=(a_0,\ldots,a_{n+1})\in  \prod_{i=0}^{n+1}(\Z/m_i\Z)\ | \  \forall i,\ a_i\neq 0,\ \sum_{i=0}^{n+1}v_ia_i=0\in \Z/d\Z\right\}.
\end{eqnarray}
\begin{theo}\label{thmweightedshioda1}

The following assertions hold:
\begin{itemize}
    \item[(i)] $V(\alpha)\neq 0$ if and only if $\alpha\in\mathfrak{A}^n_m$.
    \item[(ii)] For every $\alpha\in\mathfrak{A}^n_m$, $V(\alpha)$ is one dimensional,
    $$
V(\alpha)=\C\cdot \omega_\alpha  ,
    $$
    and it is generated by
    $$
\omega_\alpha:=\text{res}\left(\frac{x^\beta\Omega}{F^{q+1}}\right)^{p,q}\in H^{p,q}(X^n_m)_\prim  ,
    $$
    for $\beta=(\langle a_0\rangle-1,\ldots,\langle a_{n+1}\rangle-1)$ and $q=\frac{|\alpha|}{d}-1$.
    \item[(iii)] The spectral decomposition refines the Hodge decomposition as follows
    $$
H^{p,q}(X^n_m)_\prim=\bigoplus_{\alpha\in\mathfrak{A}^n_m:\  |\alpha|=d(q+1)}V(\alpha) .
    $$
\end{itemize}
\end{theo}

\begin{proof}
By \cref{thmbatycox}, $H^{p,q}(X^n_m)_\prim$ has a basis of the form
$$
\left\{\text{res}\left(\frac{x^\beta\Omega}{F^{q+1}}\right)^{p,q} \ : \ 0\le \beta_i\le m_i-2\ , \ \sum_{i=0}^{n+1}v_i\beta_i=d(q+1)-\sum_{i=0}^{n+1}v_i\right\} .
$$
Given any $g=(g_0,\ldots,g_{n+1})\in G^n_m$
$$
g^*\text{res}\left(\frac{x^\beta\Omega}{F^{q+1}}\right)^{p,q}=\text{res}\left(\frac{g^*x^\beta g^*\Omega}{F^{q+1}}\right)^{p,q}=\text{res}\left(\frac{\alpha(g)x^\beta\Omega}{F^{q+1}}\right)^{p,q}=\alpha(g)\text{res}\left(\frac{x^\beta\Omega}{F^{q+1}}\right)^{p,q}
$$
for $\alpha=(\beta_0+1,\ldots,\beta_{n+1}+1)$ and so 
$$
\text{res}\left(\frac{x^\beta\Omega}{F^{q+1}}\right)^{p,q}\in V(\alpha) .
$$
Noting that the map $\beta\mapsto \alpha$ gives a bijection with the set of all $\alpha\in\mathfrak{A}^n_m$ such that $|\alpha|=d(q+1)$, we obtain the desired result.
\end{proof}

\begin{rem}\label{remweightedshioda1}
Noting that
$$
g^*\text{res}\left(\frac{x^\beta\Omega}{F^{q+1}}\right)=\alpha(g)\text{res}\left(\frac{x^\beta\Omega}{F^{q+1}}\right) ,
$$
it follows that $\text{res}\left(\frac{x^\beta\Omega}{F^{q+1}}\right)\in V(\alpha)\subseteq H^{p,q}(X^n_m)_\prim$ and so
$$
\text{res}\left(\frac{x^\beta\Omega}{F^{q+1}}\right)=\text{res}\left(\frac{x^\beta\Omega}{F^{q+1}}\right)^{p,q}=\omega_\alpha  .
$$
\end{rem}

\begin{rem}\label{remeigenspacecyclotomic}
Another consequence of \cref{thmweightedshioda1} is that for every $\alpha\in \mathfrak{A}^n_m$ with $|\alpha|=d(q+1)$, there exists some $c_\alpha\in\C^\times$ such that
$$
V(\alpha)_{\Q(\zeta_d)}=\Q(\zeta_d)\cdot c_\alpha\omega_\alpha .
$$
In other words, $c_\alpha\omega_\alpha\in H^n(X^n_m,\Q(\zeta_d))$. In order to determine such a $c_\alpha$ we need to compute the periods of $\omega_\alpha$. We will reduce this computation to compute periods in the smooth affine variety $L=\{x_0^{m_0}+\cdots+x_{n+1}^{m_{n+1}}=1\}\subseteq\C^{n+2}$, which has $X^n_m$ as its divisor at infinity.
\end{rem}

\subsection{Periods of the spectral basis}
The computation of periods is usually done by considering topological cycles supported in affine varieties and  this motivates us discuss the de Rham cohomology of affine Fermat varieties.  
This is mainly discussed in \cite[Section 11.7]{ho13} and we reproduce it here in the framework of projective orbifolds. 

Let $X$ be a projective orbifold, and $Y_1,Y_2\subseteq X$ two quasi-smooth ample divisors intersecting each other transversely (i.e. when they intersect at a singular point $p\in X$, then $(Y_1\cup Y_2,p)$ is the quotient of the union of two smooth transverse divisors). Hence, $Y_1\cap Y_2$ is a quasi-smooth hypersurface of $Y_i$ for $i=1,2$. We have the following commutative diagram of residue maps
\[
\begin{tikzcd}
\widetilde{\Omega}_X^\bullet(\log(Y_1+Y_2)) \arrow[r, "-\text{res}"] \arrow[d, "\text{res}"'] \arrow[dr, "\text{res}"] & \widetilde{\Omega}_{Y_2}^{\bullet-1}(\log (Y_1\cap Y_2)) \arrow[d, "\text{res}"] \\
\widetilde{\Omega}_{Y_1}^{\bullet-1}(\log(Y_1\cap Y_2)) \arrow[r, "\text{res}"']                                    & \widetilde{\Omega}_{Y_1\cap Y_2}^{\bullet-2}.
\end{tikzcd}
\]
The induced diagram in hypercohomology gives a diagram in Betti cohomology with $\C$ coefficients. Since all the maps are in fact defined over $\Z$ (see \cref{remLTGZ}), we get the diagram
\begin{equation}\label{eqresdiag}
\begin{tikzcd}
H^{n+2}(X\setminus(Y_1\cup Y_2),K) \arrow[r, "-\text{res}"] \arrow[d, "\text{res}"'] & H^{n+1}(Y_2\setminus Y_1,K(-1)) \arrow[d, "\text{res}"] \\
H^{n+1}(Y_1\setminus Y_2,K(-1)) \arrow[r, "\text{res}"']                                    & H^n(Y_1\cap Y_2,K(-2))
\end{tikzcd}
\end{equation}
for any field extension $K/\Q$.

Suppose now that $X=\P^{(v,1)}$, $Y_1=\{x_0^{m_0}+\cdots+x_{n+1}^{m_{n+1}}=x_{n+2}^d\}$, $Y_2=\{x_{n+2}=0\}=\P^v$ and $Y_1\cap Y_2=X^n_m$. Let us denote by
$$
L:=Y_1\setminus Y_2=\{x\in\C^{n+2}: x_0^{m_0}+\cdots+x_{n+1}^{m_{n+1}}=1\}
$$
and by $U:=Y_2\setminus Y_1=\P^v\setminus X^n_m$. Hence, by \eqref{eqResisoK} and \eqref{eqresdiag} we get the following commutative diagram with vertical isomorphisms and horizontal epimorphisms
\begin{equation}\label{eqresdiagFermat}
\begin{tikzcd}
H^{n+2}(\C^{n+2}\setminus L,K) \arrow[r, "-\text{res}_\infty"] \arrow[d, "\text{res}"'] & H^{n+1}(U,K(-1)) \arrow[d, "\text{res}"] \\
H^{n+1}(L,K(-1)) \arrow[r, "\text{res}_\infty"']                                    & H^n(X^n_m,K(-2))_\prim.
\end{tikzcd}
\end{equation}
We denote the horizontal arrows by $\text{res}_\infty$ since we are thinking of $\P^v=\{x_{n+2}=0\}$ as the divisor at infinity of the compatification of $\C^{n+2}=\P^{(v,1)}\setminus\P^v$.
Consider now a character $\alpha=(a_0,\ldots,a_{n+1})\in \mathfrak{A}^n_m$ then, by \cref{thmweightedshioda1} and \cref{remweightedshioda1}, we have
$$
\omega_\alpha=\text{res}\left(\frac{x^\beta\Omega}{F^{q+1}}\right)\in H^n(X^n_m,\C)_\prim
$$
for $q=\frac{|\alpha|}{d}-1$, $\beta=(\langle a_0\rangle-1,\ldots,\langle  a_{m+1}\rangle-1)$, $F=x_0^{m_0}+\cdots+x_{n+1}^{m_{n+1}}$ and $\Omega=\sum_{i=0}^{n+1}(-1)^iv_ix_idx_0\wedge\cdots\widehat{dx_i}\cdots \wedge dx_{n+1}$. The following is basically \cite[Proposition 11.4]{ho13}.
\begin{prop}
Consider the form
$$
\eta_\alpha:=\frac{x^\beta dx_0\wedge dx_1\wedge\cdots \wedge dx_{n+1}}{(F-1)^{q+1}}\in H^{n+1}(\Gamma(\Omega_{\C^{n+2}\setminus L}^\bullet),d)= H^{n+1}(\C^{n+2}\setminus L,\C).
$$
Then $-\text{res}_\infty(\eta_\alpha)=\frac{x^\beta\Omega}{F^{q+1}}\in H^{n+1}(U,\C)$ and $\text{res}(\eta_\alpha)=\frac{x^\beta\Omega}{d}\in H^{n+1}(L,\C)$.
\end{prop}

\begin{proof}
In the homogeneous coordinates of $\P^{(v,1)}$ we can write
$$
\eta_\alpha=\frac{x^\beta d(\frac{x_0}{x_{n+2}^{v_0}})\wedge\cdots\wedge d(\frac{x_{n+1}}{x_{n+2}^{v_{n+1}}})}{x_{n+2}^{\sum_{i=0}^{n+1}v_i\beta_i-d(q+1)}(F-x_{n+2}^d)^{q+1}}=\frac{x^\beta dx_0\wedge\cdots\wedge dx_{n+1}}{(F-x_{n+1}^d)^{q+1}}-\frac{dx_{n+2}}{x_{n+2}}\wedge \frac{x^\beta\Omega}{(F-x_{n+1}^d)^{q+1}} ,
$$
hence
$$
-\text{res}_\infty(\eta_\alpha)=\frac{x^\beta\Omega}{F^{q+1}}\in H^{n+1}(U,\C) .
$$
In order to compute $\text{res}(\eta_\alpha)\in H^{n+1}(L,\C)$  we need to reduce first its pole order along $\{F=1\}$. To do this we use the following identity $dF\wedge\Omega=d\cdot F\cdot dx_0\wedge\cdots\wedge dx_{n+1}$ which implies
$$
d\left(\frac{x^\beta\Omega}{(F-1)^r}\right)=d(q+1-r)\frac{x^\beta dx_0\wedge\cdots\wedge dx_{n+1}}{(F-1)^r}-dr\frac{x^\beta dx_0\wedge\cdots\wedge dx_{n+1}}{(F-1)^{r+1}} ,
$$
and so in the cohomology group $H^{n+2}(\Gamma(\Omega_{\C^{n+2}\setminus L}^\bullet),d)=H^{n+2}(\C^{n+2}\setminus L,\C)$ we have the equality
$$
\frac{x^\beta dx_0\wedge\cdots\wedge dx_{n+1}}{(F-1)^{r+1}}=\left(\frac{q+1-r}{r}\right)\frac{x^\beta dx_0\wedge\cdots\wedge dx_{n+1}}{(F-1)^r} .
$$
Using the above pole order reduction, we get that
$$
\eta_\alpha=\frac{x^\beta dx_0\wedge\cdots\wedge dx_{n+1}}{F-1}=-x^\beta dx_0\wedge\cdots\wedge dx_{n+1}+\frac{1}{d}\cdot\frac{d(F-1)}{F-1}\wedge x^\beta\Omega\in H^{n+2}(\C^{n+2}\setminus L,\C)
$$
and so
$$
\text{res}(\eta_\alpha)=\frac{x^\beta\Omega}{d}\in H^{n+1}(L,\C) .
$$
\end{proof}

\begin{rem}
By \eqref{eqresdiagFermat} we see that
\begin{equation}\label{eqresres}
\text{res}_\infty(\text{res}(\eta_\alpha))=\text{res}(-\text{res}_\infty(\eta_\alpha))=\omega_\alpha .
\end{equation}
In order to determine a number $c_\alpha\in\C^\times$ such that $c_\alpha\omega_\alpha\in H^n(X^n_m,\Q(\zeta_{d}))_\prim$, it is enough to find such a constant such that $c_\alpha\eta_\alpha\in H^{n+2}(\C^{n+2}\setminus L,\Q(\zeta_d)(2))$, and this in turn is equivalent (by the isomorphism in \eqref{eqresdiagFermat}) to require that $$c_\alpha\cdot\text{res}(\eta_\alpha)\in H^{n+1}(L,\Q(\zeta_d)(1)) .$$
This reduces our problem to find $c_\alpha\in \C^\times$ such that
$$
(2\pi i)^{-1} \cdot c_\alpha x^\beta\Omega\in H^{n+1}(L,\Q(\zeta_d))=H_{n+1}(L,\Q(\zeta_d))^\vee=\hom(H_{n+1}(L,\Z),\Q(\zeta_d)) ,
$$
in other words
$$
(2\pi i)^{-1}\cdot c_\alpha\int_\delta x^\beta\Omega\in \Q(\zeta_d) \ \ \ \ , \ \forall \delta\in H_{n+1}(L,\Z) .
$$
\end{rem}

\begin{prop}
For the following number
\begin{equation}
\label{eqcalpha}
c_\alpha:=\frac{(2\pi i)(-1)^{n+1}d\prod_{i=0}^{n+1}m_i}{ B\left(\frac{\langle a_0\rangle}{m_0},\ldots,\frac{\langle a_{n+1}\rangle}{m_{n+1}}\right)}\in\C^\times  ,   \end{equation}
we have the generator of the eigenspace over $\Q(\zeta_d)$
$$
V(\alpha)_{\Q(\zeta_d)}=\Q(\zeta_d)\cdot c_\alpha\omega_\alpha .
$$
\end{prop}

\begin{proof}
It is a well-known fact that the homology of the affine variety $L=\{x_0^{m_0}+\cdots+x_{n+1}^{m_{n+1}}=1\}\subseteq\C^{n+2}$ is generated by the vanishing cycles (also known as Pham cycles) of the form
$$
\delta_{\beta'}:=\sum_{u\in\{0,1\}^{n+2}}(-1)^{\sum_{i=0}^{n+1}(1-u_i)}\Delta_{\beta'+u}\in H_{n+1}(L,\Z) ,
$$
where $\beta'\in\{0,1,\ldots,m_0-2\}\times\cdots\times\{0,1,\ldots,m_{n+1}-2\}$ and $\Delta_{\beta'+u}:\{s\in (\R_{\ge 0})^{n+2}:\sum_{i=0}^{n+1}s_i=1\}\to L$ is
$$
\Delta_{\beta'+u}(s)=(s_0^\frac{1}{m_0}\zeta_{m_0}^{\beta'_0+u_0},\ldots,s_{n+1}^\frac{1}{m_{n+1}}\zeta_{m_{n+1}}^{\beta'_{n+1}+u_{n+1}}) .
$$
see \cite[Chapter 15]{ho13} and the references therein for more details. 
By an integral computation we get the periods in terms of the Beta function
\begin{equation}\label{eqperiodsxbeta}
\int_{\delta_{\beta'}}x^\beta\Omega=\frac{(-1)^{n+1}}{\prod_{i=0}^{n+1}m_i}\cdot B\left(\frac{\langle a_0\rangle}{m_0},\ldots,\frac{\langle a_{n+1}\rangle}{m_{n+1}}\right)\cdot\prod_{i=0}^{n+1}(\zeta_{m_i}^{a_i(\beta_i'+1)}-\zeta_{m_i}^{a_i\beta_i'}) .
\end{equation}
Hence $(2\pi i)^{-1}\cdot c_\alpha\cdot\text{res}(\eta_\alpha)\in H^{n+1}(L,\Q(\zeta_d))_\prim\setminus\{0\}$, since its periods are
\begin{equation}\label{eqperiodsomegaalfa}
\int_{\delta_{\beta'}}(2\pi i)^{-1} \cdot c_\alpha\cdot\text{res}(\eta_\alpha)=\prod_{i=0}^{n+1}(\zeta_{m_i}^{a_i(\beta_i'+1)}-\zeta_{m_i}^{a_i\beta_i'}) .
\end{equation}
Finally, we get the generator of the eigenspace $V(\alpha)$ defined over $\Q(\zeta_d)$, namely $$c_\alpha\omega_\alpha=2\pi i\cdot \text{res}_\infty((2\pi i)^{-1}\cdot c_\alpha\cdot\text{res}(\eta_\alpha))\in V(\alpha)_{\Q(\zeta_d)}\subseteq H^{n}(X^n_m,\Q(\zeta_d))_\prim .$$
\end{proof}

\subsection{Galois action and Hodge characters}

Let $X$ be an orbifold, and $K/\Q$ a field extension. The Galois group $\text{Gal}(K/\Q)$ acts naturally in the cohomology groups $H^k(X,K)=H^k(X,\Q)\otimes_\Q K$. By means of the duality
$$
H^k(X,K)=H_k(X,K)^\vee=\hom(H_k(X,\Z),K).
$$
We can interpret this action in terms of homology as follows. Given $\omega\in H^k(X,K)$ and $\tau\in \text{Gal}(K/\Q)$, $\tau(\omega)\in H^k(X,K)$ is the unique cohomology class such that
$$
\tau(\omega)(\delta)=\tau(\omega(\delta)) \ , \ \ \forall \delta\in H_k(X,\Z) .
$$
In the case where $X$ is smooth and we consider 
\begin{equation}
\label{23072026sala}
\omega\in H^k(X,K)\subseteq H^k(X,\C)=H^k_\dR(X)
\end{equation}
as a differential form, we can rewrite the above relation in terms of periods as
\begin{equation}\label{eqgalactperiods}
\int_\delta\tau(\omega)=\tau\left(\int_\delta\omega\right), \ \forall \delta\in H_k(X,\Z) .
\end{equation}
Note that \eqref{23072026sala} implies that $\int_{\delta}\omega\in K$. Recall the definition of ${\mathfrak{A}^n_m}$ in \eqref{25072026tortuga}. 
\begin{prop}\label{propgalaction}
The Galois action in $H^n(X^n_m,\Q(\zeta_d))_\prim$ is compatible with the spectral decomposition. More precisely, for every $\alpha\in\mathfrak{A}^n_m$ and every $\tau\in\text{Gal}(\Q(\zeta_d)/\Q)$, of the form $\tau(\zeta_d)=\zeta_d^t$ for $t\in(\Z/d\Z)^\times$, we have
$$
\tau(V(\alpha)_{\Q(\zeta_d)})=V(t\cdot \alpha)_{\Q(\zeta_d)} .
$$
In fact, it holds that $\tau(c_\alpha\omega_\alpha)=c_{t\cdot\alpha}\omega_{t\cdot\alpha}$.
\end{prop}

\begin{proof}
By \eqref{eqperiodsomegaalfa} and \eqref{eqgalactperiods} it is clear that $\tau((2\pi i)^{-1}c_\alpha\text{res}(\eta_\alpha))=(2\pi i)^{-1}c_{t\alpha}\text{res}(\eta_{t\alpha})$.
On the other hand, since the residue map $(2\pi i)\cdot\text{res}_\infty:H^{n+1}(L,\Q(\zeta_d))\to H^n(X^n_m,\Q(\zeta_d))_\prim$ is defined over $\Q$ (see \cref{remLTGZ}), then it is Galois equivariant, thus by \eqref{eqresres} it follows that
$$
\tau(c_\alpha\omega_\alpha)=\tau(2\pi i\cdot\text{res}_\infty((2\pi i)^{-1}c_\alpha\text{res}(\eta_\alpha)))=(2\pi i)\cdot \text{res}_\infty(\tau((2\pi i)^{-1}c_\alpha\text{res}(\eta_\alpha)))
$$
$$
=(2\pi i)\cdot\text{res}_\infty((2\pi i)^{-1}c_{t\alpha}\text{res}(\eta_{t\alpha}))=c_{t\alpha}\omega_{t\alpha}.
$$
\end{proof}

\begin{rem}
The above result induces a natural Galois action in the set of characters $\text{Gal}(\Q(\zeta_d)/\Q)\simeq(\Z/d\Z)^\times\curvearrowright\mathfrak{A}^n_m$. For any $\alpha\in\mathfrak{A}^n_m$ we denote its orbit by $$O(\alpha)=\{t\cdot\alpha:t\in(\Z/d\Z)^\times\} .$$ 
\end{rem}

Using the explicit description of the Galois action given in \cref{propgalaction} we are in a position to characterize the space of Hodge cycles in terms of characters. 

\begin{theo}\label{thmweightedshioda2}
Let $X^n_m$ be a weighted Fermat variety of even dimension $n$. Then
\begin{equation}\label{eqspecdecHodgecyc}
(H^{\frac{n}{2},\frac{n}{2}}(X^n_m)_\prim\cap H^n(X^n_m,\Q))\otimes_\Q\Q(\zeta_d)=\bigoplus_{\alpha\in\mathfrak{B}^n_m}V(\alpha)_{\Q(\zeta_d)} ,
\end{equation}
where the space of Hodge characters is defined as
\begin{equation}
\label{eqHodgechar}
\mathfrak{B}^n_m:=\left\{\alpha\in\mathfrak{A}^n_m: |t\cdot\alpha|=d\left(\frac{n}{2}+1\right)\ , \ \ \forall t\in (\Z/d\Z)^\times\right\} .
\end{equation}
\end{theo}

\begin{proof}
Given any $\alpha\in\mathfrak{B}^n_m$, since the Vandermonde matrix is invertible, we get that
$$
\left\{\sum_{t\in(\Z/d\Z)^\times}\zeta_d^{t j}c_{t\alpha}\omega_{t\alpha} \ : \ j=0,\ldots,\varphi(d)-1\right\}
$$
is a $\Q(\zeta_d)$-basis of $\bigoplus_{t\in(\Z/d\Z)^\times}V(t\alpha)_{\Q(\zeta_d)}$. Moreover, each term of the basis is Galois invariant and so they belong to $H^n(X^n_m,\Q(\zeta_d))^{\text{Gal}(\Q(\zeta_d)/\Q)}=H^n(X^n_m,\Q)$. On the other hand, the condition $|t\alpha|=d(\frac{n}{2}+1)$ implies (by \cref{thmweightedshioda1}) that each term of the basis lies in $H^{\frac{n}{2},\frac{n}{2}}(X^n_m)_\prim$. This proves that the right hand side of \eqref{eqspecdecHodgecyc} is contained in the left hand side. For the other contention let $\lambda\in H^{\frac{n}{2},\frac{n}{2}}(X^n_m)_\prim\cap H^n(X^n_m,\Q)$. By \cref{thmweightedshioda1} and \cref{remeigenspacecyclotomic} we can write
$$
\lambda=\sum_{\alpha\in\mathfrak{A}^n_m\ :\  |\alpha|=d(\frac{n}{2}+1)}b_\alpha\cdot c_\alpha\omega_\alpha
$$
with $b_\alpha\in\Q(\zeta_d)$. Since $\lambda\in H^n(X^n_m,\Q)=H^n(X^n_m,\Q(\zeta_d))^{\text{Gal}(\Q(\zeta_d)/\Q)}$ is Galois invariant, we get that
$$
\tau(\lambda)=\sum_{\alpha\in\mathfrak{A}^n_m\ :\  |\alpha|=d(\frac{n}{2}+1)}\tau(b_\alpha)\cdot c_{t\alpha}\omega_{t\alpha}=\lambda\  \ \ \ , \ \forall\tau\in\text{Gal}(\Q(\zeta_d)/\Q) .
$$
It follows that if $b_\alpha\neq 0$, then $\tau(b_\alpha)=b_{t\alpha}$ and $|t\alpha|=d(\frac{n}{2}+1)$ for all $t\in(\Z/d\Z)^\times$, and so $\alpha\in\mathfrak{B}^n_m$. This proves the remaining contention.
\end{proof}

\begin{rem}\label{rembasisHodge}
As a consequence of the proof of the previous theorem we obtain an explicit basis for the space of primitive Hodge cycles $H^{\frac{n}{2},\frac{n}{2}}(X^n_m)_\prim\cap H^n(X^n_m,\Q)$, namely
$$
\left\{\lambda_{\alpha,j}:=\sum_{t\in(\Z/d\Z)^\times}\zeta_d^{tj}c_{t\alpha}\omega_{t\alpha}\ :\  O(\alpha)\in \mathfrak{B}^n_m/\text{Gal}(\Q(\zeta_d)/\Q) \ , \ j=0,\ldots,\varphi(d)-1\right\} .
$$
\end{rem}

%% file: sections/s03.tex
In this section we will introduce four operations we can perform on characters, based in geometric constructions on algebraic cycles. These operations will lead us naturally to consider the formal module of tuples, introduced and studied by Aoki \cite{Aoki1983}.
Let us start by studying the space of algebraic cycles of a weighted Fermat variety in terms of the set of algebraic characters, which we introduce as follows.

\subsection{Algebraic characters}

\begin{defi}\label{defialgchar}
Let $n\in \N$ be an even number, and $m\in \N^{n+2}$. Let $X^n_m$ be the weighted Fermat variety of degree $d=\text{lcm}(m)$. We say that a Hodge character $\gamma\in\mathfrak{B}^n_m$ is an algebraic character, if there exists some $\frac{n}{2}$-dimensional algebraic subvariety $Z\subseteq X^n_m$ such that
$$
[Z]_\prim=\sum_{\alpha\in\mathfrak{B}^n_m}b_\alpha\cdot c_\alpha\omega_\alpha\in H^{\frac{n}{2},\frac{n}{2}}(X^n_m)_\prim\cap H^n(X^n_m,\Q)
$$
and $b_\gamma\neq 0$. We denote the set of algebraic characters by $\mathfrak{C}^n_m$. We also denote by
$$
\mathfrak{C}^n_m(Z):=\{\gamma\in \mathfrak{C}^n_m: b_\gamma\neq 0\} .
$$
\end{defi}

\begin{rem}
Note that by definition $\mathfrak{C}^n_m\subseteq\mathfrak{B}^n_m$, and $$\mathfrak{C}^n_m=\bigcup_{Z\subseteq X^n_m}\mathfrak{C}^n_m(Z) .$$ Furthermore, since $[Z]_\prim$ is a rational class, then $\mathfrak{C}^n_m(Z)$ is closed under the Galois action $\text{Gal}(\Q(\zeta_d)/\Q)\curvearrowright \mathfrak{C}^n_m(Z)$.
\end{rem}

The following proposition is the key to reduce the Hodge conjecture to a question about characters.

\begin{prop}\label{propHCchar}
The weighted Fermat variety $X^n_m$ satisfies the Hodge conjecture, if and only if, $\mathfrak{C}^n_m=\mathfrak{B}^n_m$. More precisely, we have that for any $\gamma\in\mathfrak{C}^n_m$, all the Hodge cycles $\lambda_{\gamma,j}$, defined in \cref{rembasisHodge}, are algebraic cycles for $j=0,\ldots,\varphi(d)-1$.
\end{prop}

\begin{proof}
Let $\gamma\in\mathfrak{B}^n_m$, then by \cref{rembasisHodge}
$$
\lambda:=\sum_{t\in(\Z/d\Z)^\times}c_{t\gamma}\omega_{t\gamma}\in H^{\frac{n}{2},\frac{n}{2}}(X^n_m)_\prim\cap H^n(X^n_m,\Q) .
$$
If $X^n_m$ satisfies the Hodge conjecture, then $\lambda=\sum_{i=1}^kq_i[Z_i]$ for some $q_i\in \Q^\times$ and $Z_i\subseteq X^n_m$ algebraic subvarieties. Since $\lambda$ is primitive, then $\lambda=\lambda_\prim=\sum_{i=1}^kq_i[Z_i]_\prim$. Writing each $[Z_i]_\prim=\sum_{\alpha\in\mathfrak{B}^n_m}b_{i,\alpha}\cdot c_\alpha\omega_\alpha$ it follows that $1=\sum_{i=1}^kq_ib_{i,\gamma}$ and so $b_{i,\gamma}\neq 0$ for some $i$. Thus $\gamma\in\mathfrak{C}^n_m$.

Conversely, given an algebraic subvariety $Z\subseteq X^n_m$, by \cref{rembasisHodge} it is clear that
$$
[Z]_\prim\in \bigoplus_{O(\gamma)\in \mathfrak{C}^n_m(Z)/\text{Gal}(\Q(\zeta_d)/\Q)}\bigoplus_{j=0}^{\varphi(d)-1}\Q\cdot \lambda_{\gamma,j}=:V .
$$
Let us now write
$$
[Z]_\prim=\sum_{\gamma\in\mathfrak{C}^n_m(Z)}b_\gamma\cdot c_\gamma\omega_\gamma .
$$
For any $g\in G^n_m$ we have that
$$
[g^{-1}(Z)]_\prim=g^*[Z]_\prim=\sum_{\gamma\in\mathfrak{C}^n_m(Z)}b_\gamma\gamma(g)\cdot c_\gamma\omega_\gamma
$$
is also an algebraic cycle. Hence, using the spectral basis, we can compute the dimension of  space generated by these algebraic cycles 
$$
W:=\sum_{g\in G^n_m}\Q[g^{-1}(Z)]_\prim\subseteq V
$$
as the rank of the matrix
$$
M=(b_\gamma\cdot \gamma(g))_{\gamma\in \mathfrak{C}^n_m(Z), g\in G^n_m} .
$$
Since the matrix $A=(\alpha(g))_{\alpha\in\widehat{G}^n_m, g\in G^n_m}$ is invertible, and $$M=(b_\gamma\cdot \delta_{\gamma,\alpha})_{\gamma\in\mathfrak{C}^n_m(Z), \alpha\in\widehat{G}^n_m}\cdot A ,$$ then the rank of $M$ equals $\#\mathfrak{C}^n_m(Z)$ and so
$$
\dim_\Q W=\#\mathfrak{C}^n_m(Z)=\#(\mathfrak{C}^n_m(Z)/\text{Gal}(\Q(\zeta_d)/\Q))\cdot\varphi(d)=\dim_\Q V .
$$
This shows that $V=W$ and so $\lambda_{\gamma,j}$ is an algebraic cycle for all $\gamma\in \mathfrak{C}^n_m(Z)$ and $j=0,\ldots,\varphi(d)-1$. The result follows from the fact that $\bigcup_{Z\subseteq X^n_m}\mathfrak{C}^n_m(Z)=\mathfrak{C}^n_m=\mathfrak{B}^n_m$ and that $\{\lambda_{\alpha,j}: O(\alpha)\in\mathfrak{B}^n_m/\text{Gal}(\Q(\zeta_d)/\Q), j=0,\ldots,\varphi(d)-1\}$ is a basis for the space of primitive Hodge cycles.
\end{proof}

\begin{rem}\label{remgalactionalgchar}
An immediate consequence of the previous proposition is that for all $\alpha\in\mathfrak{B}^n_m$ and all $t\in(\Z/d\Z)^\times$
$$
\alpha\in\mathfrak{C}^n_m\Longleftrightarrow t\cdot\alpha\in\mathfrak{C}^n_m .
$$
\end{rem}

\begin{coro}\label{coroalgcharcyclo}
Let $\alpha\in \mathfrak{A}^n_m$, then $\alpha\in\mathfrak{C}^n_m$ if and only if $c_\alpha\omega_\alpha\in H^n(X^n_m,\Q)_\text{alg}\otimes_\Q\Q(\zeta_d)$.
\end{coro}

\begin{proof}
If $\alpha\in\mathfrak{C}^n_m$, then by \cref{propHCchar} all $\lambda_{\alpha,j}$ are algebraic for $j=0,\ldots,\varphi(d)-1$. Inverting the Vandermonde matrix it is clear that
$$
c_\alpha\omega_\alpha\in \bigoplus_{t\in (\Z/d\Z)^\times}V(t\alpha)_{\Q(\zeta_d)}=\bigoplus_{j=0}^{\varphi(d)-1}\Q(\zeta_d)\cdot \lambda_{\alpha,j}\subseteq H^n(X^n_m,\Q)_\text{alg}\otimes_\Q\Q(\zeta_d) .
$$
Conversely, if
$$
c_\alpha\omega_\alpha=\sum_{i=1}^kr_i[Z_i]_\prim
$$
for some $r_i\in\Q(\zeta_d)$, then
$$
\lambda_{\alpha,0}=\sum_{\tau\in\text{Gal}(\Q(\zeta_d)/\Q)}\tau(c_{\alpha}\omega_{\alpha})=\sum_{i=1}^kq_i[Z_i]_\prim
$$
for $q_i=\sum_{\tau\in\text{Gal}(\Q(\zeta_d)/\Q)}\tau(r_i)\in\Q$, and so $\alpha\in\mathfrak{C}^n_m$.
\end{proof}

\subsection{Lifting and permuting characters}

Let $n\in \Z_{\ge 0}$ be an even number. For $m,m'\in \N^{n+2}$ we say that $m'$ divides $m$ if $m_i'\mid m_i$ for all $i=0,\ldots,n+1$, and denote it $m'\mid m$. Let us denote by $k_i:=\frac{m_i}{m_i'}$.

\begin{defi}
The lift function
$
(\cdot)^{(m)}:\mathfrak{A}^n_{m'}\to\mathfrak{A}^n_{m}
$ is given by taking $\alpha=(a_0,\ldots,a_{n+1})\in\mathfrak{A}^n_{m'}$ to
$$
\alpha^{(m)}:= (k_0a_0,\ldots, k_{n+1}a_{n+1})\in\prod_{i=0}^{n+1}(\Z/m_i\Z) .
$$
\end{defi}

\begin{rem}
If $\ell\in\N$ is such that $m_i'\mid \ell$ for all $i=0,\ldots,n+1$, we denote by $\alpha^{(\ell)}\in\mathfrak{A}^n_{\ell}$ its lift to the homogenous Fermat variety $X^n_{\ell}$.
\end{rem}

\begin{prop}\label{propliftchar}
Let $\alpha\in\mathfrak{A}^n_{m'}$ for $m'\mid m$. We have
\begin{equation}
\alpha\in\mathfrak{B}^n_{m'} \ \Longleftrightarrow \ \alpha^{(m)}\in\mathfrak{B}^n_{m} ,
\end{equation}
\begin{equation}
\alpha\in\mathfrak{C}^n_{m'}\ \Longleftrightarrow \ \alpha^{(m)}\in \mathfrak{C}^n_{m} .    
\end{equation}
\end{prop}

\begin{proof}
For the first equivalence note that for $d:=\text{lcm}(m)$, $v_i:=\frac{d}{m_i}$, $d':=\text{lcm}(m')$, $v_i':=\frac{d'}{m_i'}$ and $t\in (\Z/d\Z)^\times$ we have
$$
|t\cdot \alpha^{(m)}|=\sum_{i=0}^{n+1}v_i\langle t\cdot k_ia_i\rangle=\sum_{i=0}^{n+1}v_ik_i\langle t\cdot a_i\rangle=\frac{d}{d'}\sum_{i=0}^{n+1}v_i'\langle t\cdot a_i\rangle=\frac{d}{d'}|t\cdot\alpha| .
$$
The result follows from the surjectivity of $(\Z/d\Z)^\times\to(\Z/d'\Z)^\times$. 

For the second equivalence consider the covering map
$$
f:X^n_{m}\to X^n_{m'}
$$
$$
(x_0:\cdots:x_{n+1})\mapsto(x_0^{k_0}:\cdots:x_{n+1}^{k_{n+1}}) .
$$
Given $Z\in\text{CH}^\frac{n}{2}(X^n_{m'})$ such that $[Z]_\prim=\sum_{\alpha\in\mathfrak{C}^n_{m'}(Z)}b_\alpha\cdot c_\alpha\omega_\alpha$, it follows by \cref{propalgcycfinitemap} that $[f^{-1}(Z)]_\prim=\frac{1}{\deg(f)}f^*[Z]_\prim=\frac{1}{\deg(f)}\sum_{\alpha\in\mathfrak{C}^n_{m'}(Z)}b_\alpha\cdot c_\alpha f^*\omega_\alpha$. Noting that $f^*\omega_\alpha=\text{res}\left(f^*\left(\frac{x^\beta\Omega}{F^{\frac{n}{2}+1}}\right)\right)=\frac{k_0\cdots k_{n+1}}{\deg(f)}\omega_{\alpha^{(m)}}$, it follows that
$$
\mathfrak{C}^n_{m}(f^{-1}(Z))=(\mathfrak{C}^n_{m'}(Z))^{(m)} .
$$
Conversely, if $\alpha\in\mathfrak{B}^n_{m'}$ is such that $\alpha^{(m)}\in\mathfrak{C}^n_{m}$ then, by \cref{propHCchar}, the Hodge cycle
$$
\lambda_{\alpha^{(m)},0}=\sum_{t\in(\Z/d\Z)^\times}c_{t\alpha^{(m)}}\omega_{t\alpha^{(m)}}\in H^{\frac{n}{2},\frac{n}{2}}(X^n_{m})_\prim\cap H^n(X^n_{m},\Q)
$$
is algebraic. Hence $\lambda_{\alpha^{(m)},0}=\sum_{i=1}^kq_i[W_i]_\prim$ for some $q_i\in\Q$ and $W_i\in\text{CH}^\frac{n}{2}(X^n_{m})$. By \eqref{eqcalpha} we have
$$
c_{t\alpha^{(m)}}=\frac{(2\pi i)(-1)^{n+1}d\prod_{i=0}^{n+1}m_i}{B\left(\frac{\langle ta_0\rangle k_0}{m_0},\ldots,\frac{\langle ta_{n+1}\rangle k_{n+1}}{m_{n+1}}\right)}=\deg(f)k_0\cdots k_{n+1}c_{t\alpha}
$$
and so
$$
\lambda_{\alpha^{(m)},0}=\deg(f)^2\sum_{t\in(\Z/d\Z)^\times}c_{t\alpha}f^*(\omega_{t\alpha})=\frac{\varphi(d)}{\varphi(d')}\deg(f)^2\cdot f^*\lambda_{\alpha,0} .
$$
Applying $f_*$ to the above equality, it follows by \cref{propalgcycfinitemap} that
$$
\lambda_{\alpha,0}=\frac{\varphi(d')}{\varphi(d)\deg(f)^3}\sum_{i=1}^kq_if_*[W_i]_\prim=\frac{\varphi(d')}{\varphi(d)\deg(f)^4}\sum_{i=1}^kq_i[f(W_i)]_\prim
$$
is an algebraic cycle, and so $\alpha\in\mathfrak{C}^n_{m'}$.
\end{proof}

The lift function corresponds with the pull-back map via
$$
f:X^n_{m}\to X^n_{m'}
$$
$$
x=(x_0:\cdots:x_{n+1})\mapsto x^k=(x_0^{k_0}:\cdots:x_{n+1}^{k_{n+1}})
$$
at the level of characters. It is natural to ask about the function corresponding to the push-forward map for characters. The map $f:X^n_{m}\to X^n_{m'}$ induces a surjective map 
$$
g=(\zeta_{m_0}^{i_0},\ldots,\zeta_{m_{n+1}}^{i_{n+1}})\in G^n_{m}\mapsto g^k=(\zeta_{m_0}^{i_0k_0},\ldots,\zeta_{m_{n+1}}^{i_{n+1}k_{n+1}})=(\zeta_{m_0'}^{i_0},\ldots,\zeta_{m_{n+1}'}^{i_{n+1}})\in G^n_{m'} .
$$
Identifying $g\in G^n_{m}\hookrightarrow\text{Aut}(X^n_{m})$ it is easy to see that
$$
f\circ g=g^k\circ f
$$
for all $g\in G^n_{m}$, and so
$$
f_*\circ g_*=g^k_*\circ f_* .
$$
Noting that $\overline{g}_*=g^*$, we obtain (after taking conjugates in the above equality) that
$$
f_*\circ g^*=(g^k)^*\circ f_*:H^n(X^n_{m},\C)_\prim\to H^n(X^n_{m'},\C)_\prim .
$$
In consequence, given $\omega_\alpha\in V(\alpha)\subseteq H^n(X^n_{m},\C)_\prim$ and $g\in G^n_{m}$ it holds
$$
(g^k)^*(f_*\omega_\alpha)=f_*(g^*\omega_\alpha)=\alpha(g)\cdot f_*\omega_\alpha  .
$$
This implies that if $f_*\omega_\alpha\neq 0$, then it is an eigenvector for $G^n_{m'}$, hence there exists some $\widetilde\alpha\in \widehat{G}^n_{m'}$ such that
$$
\widetilde\alpha(g^k)=\alpha(g)
$$
for all $g\in G^n_{m}$, i.e. $\alpha=\widetilde\alpha^{(m)}$. All these implies that 
\begin{prop}
Push-forward induces at the level of characters the inverse of the lift function, and when $\alpha\in \widehat{G}^n_{m}$ is not a lifted character, then $f_*(V(\alpha))=0$.
\end{prop}

Beside the lift, there is another operation induced by pull-back, which is the permutation. Indeed, given a permutation $\sigma\in \text{Sym}(0,1,\ldots,n+1)=\text{Sym}(n+2)$, it induces an isomorphism
$$
\sigma:X^n_m\xrightarrow{\sim} X^n_{\sigma(m)}
$$
$$
(x_0:\cdots:x_{n+1})\mapsto(x_{\sigma(0)}:\cdots:x_{\sigma(n+1)})
$$
where $\sigma(m):=(m_{\sigma(0)},\ldots,m_{\sigma(n+1)})$.

\begin{defi}
Given $\sigma\in\text{Sym}(n+2)$ and $\alpha=(a_0,\ldots,a_{n+1})\in \widehat{G}^n_m$, we define $$
\sigma(\alpha):=(a_{\sigma(0)},\ldots,a_{\sigma(n+1)})\in \widehat{G}^n_{\sigma(m)} .
$$
\end{defi}

\begin{rem}
It is easy to see that the pull-back by a permutation (and consequently its inverse, the push-forward) induces a correspondence between spectral decompositions, in fact we have
$$
\sigma^*\omega_{\sigma(\alpha)}=\omega_\alpha , \ \ \ \sigma_*\omega_\alpha=\omega_{\sigma(\alpha)} , \ \ \ c_{\sigma(\alpha)}=c_\alpha .
$$
\end{rem}

\begin{prop}\label{proppermalgchar}
Let $\alpha\in\mathfrak{A}^n_m$ and $\sigma\in \text{Sym}(n+2)$, then
\begin{equation}
\alpha\in \mathfrak{B}^n_m\Longleftrightarrow\sigma(\alpha)\in\mathfrak{B}^n_{\sigma(m)}\ ,
\end{equation}
\begin{equation}
\alpha\in \mathfrak{C}^n_m\Longleftrightarrow\sigma(\alpha)\in\mathfrak{C}^n_{\sigma(m)} \ .
\end{equation}
\end{prop}

\begin{proof}
The first assertion follows from $|\sigma(\alpha)|=|\alpha|$, while the second follows from $\sigma_*[W]=[\sigma(W)]$ and $\sigma^*[Z]=[\sigma^{-1}(Z)]$ for any $W\in \text{CH}^\frac{n}{2}(X^n_m)$ and $Z\in\text{CH}^\frac{n}{2}(X^n_{\sigma(m)})$, which hold by \cref{propalgcycfinitemap}.
\end{proof}

\subsection{Joining and deleting characters}
The join of polynomials has first appeared in singularity theory, see \cite{arn, ho13} and it has inspired the algebraic join  in \cite{franco2023periods}. In this section we explain this in the level of characters. 
\begin{defi}
Given two characters $\alpha=(a_0,\ldots,a_{n+1})\in\widehat{G}^n_m$ and $\alpha'=(a_0',\ldots,a_{n'+1}')\in\widehat{G}^{n'}_{m'}$, we define their join
$$
\alpha*\alpha':=(a_0, \ldots,a_{n+1},a_0',\ldots,a_{n'+1}')\in \widehat{G}^{n+n'+2}_{(m, m')} .
$$
\end{defi}

\begin{prop}\label{propjoinalgchar}
Let $\alpha\in\widehat{G}^n_m$ and $\alpha'\in\widehat{G}^{n'}_{m'}$, then following assertions hold.
\begin{itemize}
    \item[(i)] If $\alpha\in \mathfrak{B}^n_m$, then
    $\alpha'\in\mathfrak{B}^{n'}_{m'} \Longleftrightarrow \alpha * \alpha'\in \mathfrak{B}^{n+n'+2}_{(m,m')}.$
    \item[(ii)] If $\alpha\in \mathfrak{C}^n_m$, then
    $\alpha'\in\mathfrak{C}^{n'}_{m'} \Longleftrightarrow \alpha * \alpha'\in \mathfrak{C}^{n+n'+2}_{(m,m')}.$
\end{itemize}
\end{prop}

\begin{proof}
Let $d'':=\text{lcm}(m,m')=\text{lcm}(d,d')$. The first assertion follows from $|\alpha*\alpha'|=\frac{d''}{d}|\alpha|+\frac{d''}{d'}|\alpha'|$. For the second assertion, note first that by \cref{propliftchar} we can reduce ourselves to the homogeneous same degree case, because $(\alpha*\alpha')^{(d'')}=\alpha^{(d'')}*\alpha'^{(d'')}$. So, from now on we assume $X^n_m=X^n_d$ and $X^{n'}_{m'}=X^{n'}_d$, i.e. both are homogeneous degree $d$ Fermat varieties. Consider two algebraic cycles $Z\in\text{CH}^\frac{n}{2}(X^n_d)$ and $Z'\in\text{CH}^\frac{n'}{2}(X^{n'}_d)$ such that
$$
[Z]_\prim=\sum_{\alpha\in\mathfrak{C}^{n}_d(Z)}b_\alpha c_\alpha\omega_\alpha ,
$$
$$
[Z']_\prim=\sum_{\alpha'\in\mathfrak{C}^{n'}_d(Z')}b_{\alpha'} c_{\alpha'}\omega_{\alpha'} .
$$
By \cite[Theorem 1.2]{franco2023periods} we have that
$$
[J(Z,Z')]_\prim=\frac{(\frac{n+n'+2}{2})!\cdot d}{(2\pi i)\frac{n}{2}!\cdot\frac{n'}{2}!}\sum_{\alpha\in \mathfrak{C}^n_d(Z)}\sum_{\alpha'\in\mathfrak{C}^{n'}_d(Z')}b_\alpha c_\alpha b_{\alpha'}c_{\alpha'}\omega_{\alpha*\alpha'}
$$
where $J(Z,Z')\in\text{CH}^{\frac{n+n'+2}{2}}(X^{n+n'+2}_d)$ is the join of $Z$ and $Z'$. It follows that
$$
\mathfrak{C}^{n+n'+2}_d(J(Z,Z'))=\mathfrak{C}^n_d(Z)*\mathfrak{C}^{n'}_d(Z') .
$$
This proves the first implication. For the converse, if $\alpha\in \mathfrak{C}^n_d$ and $\alpha*\alpha'\in\mathfrak{C}^{n+n'+2}_d$, then by \cref{coroalgcharcyclo} we can write
$$
c_\alpha\omega_\alpha=[Z]  , \ \ \ \ \ 
c_{\alpha *\alpha'}\omega_{\alpha *\alpha'}=[W]  ,
$$
for some $Z\in\text{CH}^\frac{n}{2}(X^n_d)\otimes_\Z\Q(\zeta_d)$ and $W\in\text{CH}^\frac{n+n'+2}{2}(X^{n+n'+2}_d)\otimes_\Z\Q(\zeta_d)$. Let $\delta\in H_{n'}(X^{n'}_d,\Q(\zeta_d))$ such that $\delta^{pd}=c_{\alpha'}\omega_{\alpha'}\in H^{n'}(X^{n'}_d,\Q(\zeta_d))$. Then 
$$
[W]=c_{\alpha*\alpha'}\omega_{\alpha *\alpha'}=\frac{(2\pi i)\frac{n}{2}!\cdot\frac{n'}{2}!\cdot c_{\alpha *\alpha'}}{(\frac{n+n'+2}{2})!\cdot d\cdot c_\alpha c_{\alpha'}}[J(Z,\delta)]=-[J(Z,\delta)]
$$
and so
$$
[J(pt,\delta)]=\frac{1}{\deg(Z\cap Z)}[J(Z,\delta)\cap J(Z,X^{n'}_d)]=\frac{-1}{\deg(Z\cap Z)}[W\cap J(Z,X^{n'}_d)]
$$
is an algebraic cycle with $\Q(\zeta_d)$ coefficients (since $Z\cap Z\in\text{CH}^n(X^n_d)\otimes_\Z\Q(\zeta_d)$), i.e. lies in $H^{2n+n'+2}(X^{n+n'+2}_d,\Q)_\text{alg}\otimes_\Q\Q(\zeta_d)$. Using the inductive structure of homogeneous Fermat varieties (see \cite[Theorem II]{sh79}) one can show that this is equivalent to have $(pt\times\delta)^{pd}\in H^{2n+n'}(X^n_d\times X^{n'}_d,\Q)_\text{alg}\otimes_\Q\Q(\zeta_d)$, which in turn is equivalent to $\delta^{pd}\in H^{n'}(X^{n'}_d,\Q)_\text{alg}\otimes_\Q\Q(\zeta_d)$. Again by \cref{coroalgcharcyclo} we conclude that $\alpha'\in\mathfrak{C}^{n'}_d$.
\end{proof}

\begin{rem}
The previous proposition asserts that both the Hodge property and the algebraic property for characters are preserved under the join operation. Moreover, it also asserts that this operation is reversible in each class.
\end{rem}

\begin{defi}
Let $\alpha\in\mathfrak{A}^n_m$, $\alpha'\in\mathfrak{A}^{n'}_{m'}$ and $\gamma:=\alpha *\alpha'$. We say that $\alpha'$ is obtained by deleting $\alpha$ from $\gamma$.
\end{defi}

\begin{rem}
The previous proposition asserts that deleting a Hodge character from a Hodge character remains Hodge, and deleting an algebraic character from an algebraic character remains algebraic.
\end{rem}

\subsection{Linear and standard characters}

Up to now we have described four operations useful to produce algebraic characters, namely the lift, permutation, join and deletion. But still we have not shown an explicit algebraic character. In this section we recall the classical linear (or decomposable) characters introduced Ran \cite{Ran1980} and Shioda \cite{sh79}, and the standard characters introduced by Aoki \cite{aoki1987}.

\begin{defi}
Let $n\in\N$ even, and any $d\in\N$. The space of linear characters is defined as
$$
    \mathfrak{D}^n_d:=\left\{\alpha\in\mathfrak{A}^n_d: \ \exists \sigma\in\text{Sym}(n+2)\ \text{ s.t. }\ \langle a_{\sigma(2i)}\rangle+\langle a_{\sigma(2i+1)}\rangle=d \ , \ \forall i=0,\ldots,\frac{n}{2} \right\}.
$$
\end{defi}

\begin{prop}\label{proplinchar}
All linear characters are algebraic, i.e. $\mathfrak{D}^n_d\subseteq\mathfrak{C}^n_d$. In fact, for every $\alpha\in\mathfrak{D}^n_d$, there exists a linear subvariety $L\subseteq X^n_d$ of dimension $\frac{n}{2}$ such that $\alpha\in \mathfrak{C}^n_d(L)$.
\end{prop}

\begin{proof}
Consider the linear cycle
$$
L=\{x_0-\zeta_{2d}x_1=x_2-\zeta_{2d}x_3=\cdots=x_n-\zeta_{2d}x_{n+1}=0\}\subseteq X^n_d .
$$
By \cite[Corollary 8.3]{RobertoThesis}
$$
[L]_\prim=\frac{(-1)^{\frac{n}{2}+1}d^\frac{n}{2}\left(\frac{n}{2}\right)!}{(2\pi i)^\frac{n}{2}}\sum_{i_0=1}^{d-1}\sum_{i_1=1}^{d-1}\cdots\sum_{i_\frac{n}{2}=1}^{d-1}\zeta_{2d}^{\sum_{j=0}^\frac{n}{2}i_j}\omega_{(d-i_0,i_0,d-i_1,i_1,\ldots,d-i_\frac{n}{2},i_\frac{n}{2})}  .
$$
This shows that
$$
\left\{\alpha=(a_0,\ldots,a_{n+1})\in\mathfrak{A}^n_d: \ \ \langle a_{2i}\rangle+\langle a_{2i+1}\rangle=d \ , \ \forall i=0,\ldots,\frac{n}{2} \right\}\subseteq\mathfrak{C}^n_d .
$$
Since the rest of the terms of $\mathfrak{D}^n_d$ are obtained by the action of $\text{Sym}(n+2)$, the result follows from \cref{proppermalgchar}.
\end{proof}

\begin{rem}
If one permits to consider $0$-dimensional Fermat varieties, then every element of $\mathfrak{D}^n_d$ is of the form $\sigma(\delta_0*\delta_1*\cdots*\delta_{\frac{n}{2}})$ for $\sigma\in\text{Sym}(n+2)$ and $\delta_i\in\mathfrak{D}^0_d$ for all $i=0,\ldots,\frac{n}{2}$, and so the previous proposition follows from the equality $$
\mathfrak{B}^0_d=\mathfrak{C}^0_d=\mathfrak{D}^0_d.
$$
\end{rem}

\begin{theo}[\cite{Aoki1983}]\label{theoAokilinearcycles}
Let $n\in \N$ even and $d\in \N$, then $\mathfrak{B}^n_d=\mathfrak{D}^n_d$, if and only if, $d=4$ or $d$ is prime or $\gcd(d,(n+2)!)=1$.
\end{theo}

\begin{defi}\label{defistandardchar}
Let $p\in\N$ be a prime number and $d=pk$, $k\in\N$ such that $k\ge 3$. For each $i\in\Z/d\Z$ such that $p\nmid i$, we define the $p$-standard characters
$$
s_{p,i}:=\left\{\begin{array}{cc}
    (i,k+i, 2k+i,\ldots,(p-1)k+i,d-pi)\in\mathfrak{A}^{p-1}_d & \text{ if }p\ge 3 , \\
    (i,k+i,d-2i,k)\in\mathfrak{A}^2_d & \text{ if }p=2 .
\end{array}\right.
$$
We denote the set of standard characters by
$$
\mathfrak{S}^n_d:=\left\{\begin{array}{cc}
    \{s_{p,i}:p\nmid i\} & \text{ if }n+1=p\ge 5\text{ is prime}  , d=pk, k\ge 3, \\
    \{s_{3,i}:3\nmid p\} & \text{ if }n=2, d=3k, k>1\text{ odd}, \\\{s_{2,i}:2\nmid i\} & \text{ if }n=2,d=6 \text{ or } d=2k, k\ge 4, 3\nmid k, \\\{s_{2,i}:2\nmid i\}\cup\{s_{3,i}:3\nmid i\} & \text{ if }n=2, d=6k, k>1, \\
    \varnothing & \text{ otherwise}.
\end{array}\right.
$$
\end{defi}

\begin{defi}
Let $p\in\N$ be a prime number and $d=pk$ for $k\ge 3$. We define the Aoki-Shioda projective variety $A_{p,d}$ by the equations
$$
A_{2,d}:=\left\{
    x_0^k+x_1^k+ix_2^k=x_3^2-\sqrt[k]{2}x_0x_1=0 \right\}\subseteq\P^3 ,
$$
while for $p\ge 3$ we define
$$
A_{p,d}:=\left\{\begin{array}{cc}
    x_0^{kj}+x_1^{kj}+\cdots+x_{p-1}^{kj}=0 & \text{ for }j=1,\ldots,\frac{p-1}{2}, \\
    x_p^p-\zeta_{2k}\sqrt[k]{p}x_0x_1\cdots x_{p-1}=0 & 
\end{array}\right\}\subseteq\P^p.
$$
\end{defi}

\begin{rem}
It is elementary to show that $A_{2,d}\subseteq X^2_d$ and that $A_{p,d}\subseteq X^{p-1}_d$ for $p\ge 3$ (c.f. \cite{aoki1987} or \cite[\S 17.5]{ho13}).
\end{rem}

\begin{prop}\label{propstandchar}
The standard characters are algebraic, i.e.
$\mathfrak{S}^n_d\subseteq\mathfrak{C}^n_d$.
\end{prop}

\begin{proof}
Let $p\in\N$ prime, $d=pk$, $k\ge 3$ and $i\in\Z/d\Z$ such that $p\nmid i$. Write $i=j\ell$ with $\text{gdc}(i,d)=\text{gcd}(i,k)=\ell$, then
$$
s_{p,i}=(s_{p,j})^{(d)}
$$
where $s_{p,j}\in\mathfrak{A}^{n}_\frac{d}{\ell}$ and $\text{gcd}(j,\frac{d}{\ell})=1$. Thus, by \cref{propliftchar} we reduce ourselves to show that $s_{p,i}\in\mathfrak{C}^n_d$ for $i\in(\Z/d\Z)^\times$. On the other hand, for $i\in(\Z/d\Z)^\times$
$$
s_{p,i}=i\cdot s_{p,1} ,
$$
and so by \cref{remgalactionalgchar} it is enough for us to show that $s_{p,1}\in\mathfrak{C}^n_d(A_{p,d})$. Since the Aoki-Shioda varieties are complete intersections, we can apply \cite[Theorem 1.1]{RobertoThesis} to compute $[A_{p,d}]_\prim$ and check that $\omega_{s_{p,1}}$ appears in its expansion. This computation is quite tedious, so we skip it and refer the reader to \cite{aoki1987} for a complete proof with a different method.
\end{proof}

\begin{rem}
By Lefschetz $(1,1)$ theorem we always have that $\mathfrak{B}^2_d=\mathfrak{C}^2_d$, but in general we do not know explicit algebraic representatives of a given Hodge cycle in a Fermat surface. Moreover, by a result of Kang \cite[Corollary 3.2]{kang2016refined} the Hodge conjecture holds for Fermat fourfolds, i.e. $\mathfrak{B}^4_d=\mathfrak{C}^4_d$. In fact, she shows that the generalized Hodge conjecture holds for Fermat fourfolds.
\end{rem}

\subsection{Algebraic tuples of odd length}

The definition of Hodge characters given in \eqref{eqHodgechar} also makes sense for $n$ odd, when the degree $d=\text{lcm}(m)$ is an even number. In spite that these tuples do not match with an even dimensional weighted Fermat, it is still possible to join them to construct Hodge or even algebraic characters. For this reason, we also want to consider such tuples.

\begin{defi}
Let $n\in\Z_{\ge -1}$, $m\in\N^{n+2}$ and $d:=\text{lcm}(m)$. We define
$$
\mathfrak{A}^n_m:=\prod_{i=0}^{n+1}\{1,2,\ldots, m_i-1\},
$$ 
and the set of Hodge tuples
$$
\mathfrak{B}^n_m:=\left\{\alpha\in\mathfrak{A}^n_m: |t\cdot \alpha|=d\left(\frac{n}{2}+1\right)\ ,\ \ \forall t\in(\Z/d\Z)^\times\right\}.
$$
\end{defi}

\begin{rem}
We extend in the obvious way the lift, permutation, join and deletion operations to tuples. And the same arguments as before show that these operations preserve Hodge tuples.
\end{rem}

\begin{defi}
Let $d\in\N$ be an even number. Then we define
$$
\varepsilon_d:=\left(\frac{d}{2}\right)\in\mathfrak{B}^{-1}_d.
$$
Note that $\mathfrak{B}^{-1}_d=\{\varepsilon_d\}$.
\end{defi}

\begin{rem}
For $n\in\N$ even, the set $\mathfrak{B}^n_m$ is the usual set of Hodge characters. For $n$ odd, we have $\mathfrak{B}^n_m=\varnothing$ if $d$ is also odd. On the other hand, if $n$ is odd and $d$ is even, then for all $\alpha\in\mathfrak{A}^n_m$
\begin{equation}
\label{eqjoinepsilonHodge}
\alpha\in\mathfrak{B}^{n}_m\Longleftrightarrow\alpha *\varepsilon_d\in\mathfrak{B}^{n+1}_{(m,d)}.
\end{equation}
This observation leads us to define the following extension of the notion of algebraic characters to tuples of odd size.
\end{rem}

\begin{defi}
Let $n\in\Z_{\ge -1}$ odd, $m\in\N^{n+2}$ and $d:=\text{lcm}(m)$. We define the set of algebraic tuples
$$
\mathfrak{C}^n_m:=\left\{\alpha\in\mathfrak{A}^n_m: \alpha*\varepsilon_d\in\mathfrak{C}^{n+1}_{(m,d)}\right\}.
$$
\end{defi}

\begin{rem}
Since $\mathfrak{C}^{n+1}_{(m,d)}\subseteq\mathfrak{B}^{n+1}_{(m,d)}$, it follows by \eqref{eqjoinepsilonHodge} that
$$
\mathfrak{C}^n_m\subseteq\mathfrak{B}^n_m.
$$
Noting that $\varepsilon_d*\varepsilon_d=(\frac{d}{2},\frac{d}{2})\in\mathfrak{D}^0_d=\mathfrak{C}^0_d$ it follows that  $$\varepsilon_d\in\mathfrak{C}^{-1}_d.$$
\end{rem}

\begin{prop}
The lift, permutation, join and deletion preserve algebraic tuples.
\end{prop}

\begin{proof}
The assertion about the lift follows from \cref{propliftchar} and the fact that $(\varepsilon_{d'})^{(d)}=\varepsilon_{d}$ for $d'\mid d$. The assertion about permutations follows from \cref{proppermalgchar} and the fact that $\sigma(\alpha)*\varepsilon_d=(\sigma,id_{\{n+2\}})(\alpha*\varepsilon_d)$. For the join and deletion let $\alpha\in\mathfrak{C}^n_m$ and $\alpha'\in\mathfrak{A}^{n'}_{m'}$. We want to show that
$$
\alpha'\in\mathfrak{C}^{n'}_{m'}\Longleftrightarrow\alpha*\alpha'\in\mathfrak{C}^{n+n'+2}_{(m,m')}.
$$
When $n$ and $n'$ are even, this is \cref{propjoinalgchar}. Using \cref{propjoinalgchar}, the fact that $\varepsilon_d*\varepsilon_d\in\mathfrak{C}^0_d$, and the fact that algebraic tuples are preserved under lift and permutations, we will show the equivalence in the remaining cases. If $n$ is even and $n'$ is odd, then for $d'=\text{lcm}(m')$ and $d''=\text{lcm}(m,m')$ 
$$
\alpha'\in\mathfrak{C}^{n'}_{m'}\Leftrightarrow \alpha'*\varepsilon_{d'}\in\mathfrak{C}^{n'+1}_{(m',d')}\Leftrightarrow \alpha*\alpha'*\varepsilon_{d'}\in\mathfrak{C}^{n+n'+3}_{(m,m',d')}
$$ 
$$
\Leftrightarrow \alpha*\alpha'*\varepsilon_{d''}\in\mathfrak{C}^{n+n'+3}_{(m,m',d'')}\Leftrightarrow\alpha*\alpha'\in\mathfrak{C}^{n+n'+2}_{(m,m')}.
$$
If $n$ is odd and $n'$ is even, then
$$
\alpha'\in\mathfrak{C}^{n'}_{m'}\Leftrightarrow \alpha*\varepsilon_d*\alpha'\in\mathfrak{C}^{n+n'+3}_{(m,d,m')}\Leftrightarrow\alpha*\alpha'*\varepsilon_{d''}\in\mathfrak{C}^{n+n'+3}_{(m,m',d'')}\Leftrightarrow \alpha*\alpha'\in\mathfrak{C}^{n+n'+2}_{(m,m')}.
$$
Finally, if $n$ and $n'$ are odd, then
$$
\alpha'\in\mathfrak{C}^{n'}_{m'}\Leftrightarrow \alpha'*\varepsilon_{d'}\in\mathfrak{C}^{n'+1}_{(m',d')}\Leftrightarrow \alpha*\varepsilon_d*\alpha'*\varepsilon_{d'}\in\mathfrak{C}^{n+n'+4}_{(m,d,m',d')}
$$
$$
\Leftrightarrow \varepsilon_{d''}*\varepsilon_{d''}*\alpha*\alpha'\in\mathfrak{C}^{n+n'+4}_{(d'',d'',m,m')}\Leftrightarrow \alpha*\alpha'\in\mathfrak{C}^{n+n'+2}_{(m,m')}.
$$
\end{proof}

%% file: sections/s04.tex
The operations on characters described in the previous section lead us to work with the space of all tuples of any length and degree at the same time, considering the join as a sum and the deletion as a subtraction. Furthermore the invariance under permutation and lift (of the sets of algebraic and Hodge tuples) suggests that we should identify the reorderings and the lifts of a given tuple. In this section, we introduce the formal module of tuples and recall its main properties discovered by Aoki \cite{Aoki1983}. Using such properties we reduce the Hodge conjecture to a question about lengths of exceptional tuples. We propose a new length reduction algorithm whose main novelty is that it incorporates the lift operation in the reduction process. Using it we verify the Hodge conjecture for all weighted Fermat varieties of degree $<44$ excepting degree $35$. In a subsequent article, we will  use the formal module of tuples to compute Gamma products, which are needed in order to look for explicit equations of new algebraic cycles.

\subsection{The formal module $R$}\label{sec:4.1}

\begin{defi}
Given $d > 1$, let $R_d$ be the free $\mathbb{Z}$-module generated by $(\Z/d\Z)\setminus\{0\}$, i.e. by the formal symbols $(a)=(a)_d$ for each $a \in (\mathbb{Z}/d\mathbb{Z}) \setminus \{0\}$. We have a natural action of the group $G_d := (\mathbb{Z}/d\mathbb{Z})^\times$ on $R_d$, given by
$$    
g\cdot(a) :=(ga).
$$
In this way $R_d$ becomes a $G_d$-module.
\end{defi}

\begin{rem}\label{remchartotuples}
We have a natural map (of sets)
$$\alpha=(a_0, \ldots, a_{n+1})\in\mathfrak{A}^n_d \mapsto\sum_{i=0}^{n+1} (a_i)\in R_d
$$
whose image determines the tuple $\alpha$ up to permutations. Thus we have an injection
\[\mathfrak{A}_d^n/\text{Sym}(n+2) \;  \hookrightarrow R_d.\]
From now on we will use the above injection to identify tuples (up to permutation) with elements in $R_d$ 
$$
\alpha= \sum_{i=0}^{n+1}(a_i).
$$
Conversely, given $\alpha = \sum c_a (a) \in R_d$ with $c_a \geq 0$, we often denote $\alpha$ by a tuple $(..., a, ...)$ in which each $a$ appears $c_a$ times. For example, $(1, 1, 2) = 2(1) + (2)$.
For two tuples $\alpha,\alpha'\in\mathfrak{A}^n_d$ such that $\alpha=\alpha'$ as elements of $R_d$, we will denote $\alpha\sim \alpha'$. In other words, $\alpha\sim\alpha'$ if $\alpha'$ is a permutation of $\alpha$. Under this identification the join operation becomes simply the sum
$$
\alpha*\alpha'=\sum_{i=0}^{n+1}(a_i)+\sum_{j=0}^{n'+1}(a_j')=\alpha+\alpha'
$$
and the deletion operation becomes a subtraction
$$
\alpha'=(\alpha*\alpha')-\alpha.
$$
On the other hand, the lift map $(\cdot)^{(d)}:\mathfrak{A}^n_{d'}\to\mathfrak{A}^n_{d}$ for $d'\mid d$ and $k=\frac{d}{d'}$, induces a group monomorphism
$$
(\cdot)^{(d)}:R_{d'}\hookrightarrow R_{d}
$$
$$
\sum_{i=1}^\ell n_i(a_i)_{d'}\mapsto \sum_{i=1}^\ell n_i(ka_i)_{d}.
$$
These morphisms form a directed system of groups.
\end{rem}

\begin{defi}
The formal module of tuples $R$ is the union of all $R_d$ modulo lifting  $$R:=\varinjlim R_{d}.$$
\end{defi}

\begin{rem}
We will identify $R_d\hookrightarrow R$ with its image in $R$, and so we will have $R_{d'}\subseteq R_d$ for $d'\mid d$, $k=\frac{d}{d'}$ and we will identify a tuple with its lift
$$
(ka)_{d}=(a)_{d'}.
$$
Under these identifications we see that formal module of tuples $R$ is isomorphic to the free $\Z$-module generated by $(\Q/\Z)\setminus\{0\}$ under the identification $(a)_d= ({a}/{d})_1$.
\end{rem}

\begin{defi}
We define the module of Hodge tuples $B\subseteq R$ as the submodule generated by all $\mathfrak{B}^n_d/\sim\ \hookrightarrow R$ for all $n\ge -1$ and $d\in\N$. Similarly, we define the module of algebraic tuples $C\subseteq R$ as the submodule generated by all $\mathfrak{C}^n_d/\sim\ \hookrightarrow R$ for all $n\ge -1$ and $d\in\N$. We denote by $B_d:=B\cap R_d$ and $C_d:=C\cap R_d$.
\end{defi}

\begin{rem}
The contention $\mathfrak{C}^n_d\subseteq \mathfrak{B}^n_d$ implies that $C\subseteq B$ and so $C_d\subseteq B_d$.
\end{rem}

\begin{rem}
Since the sets of Hodge and algebraic characters are invariant under lifts (\cref{propliftchar}), it follows that $B_d\subseteq R_d$ is the submodule generated by $\mathfrak{B}^n_d/\sim\ \hookrightarrow R_d$ for all $n\ge -1$, and $C_d\subseteq R_d$ is the submodule generated by $\mathfrak{C}^n_d/\sim\ \hookrightarrow R_d$ for all $n\ge -1$.
\end{rem}

\begin{theo}\label{thm:reduce_hodge_conjecture_to_Bm} If $C_d = B_d$, then the Hodge conjecture is true for all weighted Fermat varieties of degree $d$.
\end{theo}
\begin{proof}
If for some $n\in\N$ even and some $m\in\N^{n+2}$ such that $\text{lcm}(m)=d$ we have $\alpha\in\mathfrak{B}^n_m\setminus\mathfrak{C}^n_m$, then $\alpha\in B_d$ and we claim $\alpha\notin C_d$. In fact, if we suppose that $\alpha\in C$, then $\alpha=\sum_{i=1}^\ell k_i\alpha_i$ for some $\alpha_i\in\mathfrak{C}^{n_i}_{d_i}$ and $k_i\in\Z\setminus\{0\}$. Let $d':=\text{lcm}(d,d_1,\ldots,d_\ell)$. Without loss of generality we can assume that $k_i>0$ for all $i=1,\ldots,r$ and $k_j=-h_j<0$ for all $j=r+1,\ldots,\ell$. Then, we can write in $R_{d'}$ the equality
$$
\alpha^{(d')}+\sum_{j=r+1}^\ell h_j\cdot \alpha_j^{(d')}=\sum_{i=1}^rk_i\cdot\alpha_i^{(d')} .
$$
This equality implies that as characters
$$
\alpha^{(d')}*(\alpha_1^{(d')})^{*h_1}*\cdots*(\alpha_r^{(d')})^{*h_r}\sim(\alpha_{r+1}^{(d')})^{*k_{r+1}}*\cdots*(\alpha_\ell^{(d')})^{*h_\ell},
$$
where we use the notation $\gamma^{*k}=\gamma*\gamma*\cdots*\gamma$ for the join of $k$ copies of $\gamma$. And it follows from \cref{propliftchar}, \cref{proppermalgchar} and \cref{propjoinalgchar} that $\alpha\in \mathfrak{C}^n_m$ which is a contradiction.
\end{proof}

\begin{rem}
For $d'\mid d$ we have $C_{d'}=C_{d}\cap R_{d'}$ and $B_{d'}=B_{d}\cap R_{d'}$. Thus, if $C_d=B_d$, then the Hodge conjecture is true for all weighted Fermat varieties of degree a divisor of $d$.
\end{rem}

The previous result motivates us to investigate the $G_d$-module structure of $B_d$. A generating set for $B_d$ was described by Aoki~\cite{Aoki1983}, and we will explain it now.

\begin{defi}
We define the module of linear tuples $D\subseteq R$ as the submodule generated by all $\mathfrak{D}^n_d/\sim \ \hookrightarrow R$ for all $n\ge -1$ and $d\in\N$. We denote by $D_d:=D\cap R_d$.
\end{defi}

\begin{rem}
As a $G_d$-module, the module $D_d\subseteq R_d$ is generated by the tuple $(1,-1)$ if $d$ is odd, while for $d$ even it is generated by $(1,-1)$ and $\varepsilon_d=(\frac{d}{2})$.
\end{rem}

\begin{prop}
Let $d\in \N_{>1}$. Then $B_d=D_d$, if and only if, $d$ is prime or $d=4$.
\end{prop}

\begin{proof}
If $B_d=D_d$, then $\mathfrak{B}^n_d=\mathfrak{D}^n_d$ for all $n\in\N$ even and the result follows by \cref{theoAokilinearcycles}.
\end{proof}

\begin{defi}
For every $d\in\N$ we define the module of standard tuples $S_d\subseteq R_d$ as the $G_d$-submodule generated by the images of $\mathfrak{S}^n_d\rightarrow R$ for all $n\ge -1$. We denote by $\widetilde{S}\subseteq R$ the submodule generated by all $S_d$ for $d\in \N$, and by $\widetilde{S}_d:=\widetilde{S}\cap R_d$.
\end{defi}

\begin{rem}
By \cref{proplinchar} and \cref{propstandchar} we have $D+\widetilde{S}\subseteq C$ and so 
\begin{equation}\label{eqlinstdalg}
D_d+S_d\subseteq D_d+\widetilde{S}_d\subseteq C_d.
\end{equation}
Furthermore, it is clear from the definition of $\mathfrak{D}^n_d$ that $D_d\subseteq R_d$ is the submodule generated by $\mathfrak{D}^n_d/\sim\ \hookrightarrow R_d$ for all $n\ge -1$. On the other hand, the analogue statement is not true anymore for standard tuples, i.e. it might happen that $S_d\subsetneq \widetilde{S}_d$. 
The main difference between the two is the lift operation. If $\alpha\in\mathfrak{A}^n_d$ is such that $\alpha\in S_d$, then it corresponds to a character resulting from joining, permuting and deleting standard characters of degree $d$, while if $\alpha\in\widetilde{S}_d$, then it is a character resulting from joining, permuting and deleting lifts of standard characters (not necessarily of the same degree). The advantage of working with $\widetilde{S}_d$ is that it produces more algebraic characters than $S_d$, but it has the disadvantage of not knowing an explicit finite set of generators of it.
\end{rem}

\subsection{The $G_d$-module of exceptional tuples $T_d$}\label{sec:4.2}

In view of \cref{thm:reduce_hodge_conjecture_to_Bm} and \eqref{eqlinstdalg}, we may therefore focus on the simpler problem of finding generators of the quotient $B / (D + \widetilde{S})$. Aoki \cite[Theorem C]{Aoki1983} was in fact able to find generators for $B_d / (D_d + S_d)$, the so called exceptional tuples. In this section we recall the definition of such generators. We must first set some notation.

\begin{defi}
Given $d\in\N$, $\alpha = \sum_{i=1}^\ell n_i (a_i)_d \in R_d$ and a proper divisor $d' \mid d$, we set the $d'$-part of $\alpha$ as
\begin{equation}
   \alpha_{d'} := \sum_{ \text{gcd}(a_i, d) = d'} n_i (a_i/d')_{d/d'} \in R_{d/d'}.
\end{equation}
\end{defi}

\begin{rem}
Any $\alpha\in R_d$ can be decomposed uniquely as the sum of $d'$-parts, 
$$\alpha = \sum_{d'\mid d} \alpha_{d'}.$$ We call the $1$-part of $\alpha$ its primitive part. For example, for a tuple $(3, 3, 4, 6, 7, 10)\in \mathfrak{A}^4_{12}$
\[(3, 3, 4, 6, 7, 10)_{12}=(7)_{12}+(10)_{12}+(3,3)_{12}+(4)_{12}+(6)_{12} \]
\[= (7)_{12} + (5)_6 + (1, 1)_4 + (1)_3 + (1)_2 \in R_{12},\]
and the primitive part is $(7)_{12} \in R_{12}$.
\end{rem}


\begin{defi}
Let $d,e\in\N$ coprime. Given $a\in(\Z/d\Z)\setminus\{0\}$ and $b\in(\Z/e\Z)\setminus\{0\}$, we define
$$
(a)_d\cdot (b)_e:=(c)_{de}
$$
where $c$ is uniquely determined by the equations $c\equiv a$ (mod $d$) and $c\equiv b$ (mod $e$).
\end{defi}


Let us denote the following modified set of prime numbers by
$$
P=\{4\}\cup\{p\in\N_{\ge 3}:p\text{ prime}\} 
$$
and let us denote the set of all products of an even number of distinct elements in $P$ by
$$
Q =\{p_1p_2\cdots p_{2k}: k\in\N,\  p_i\in P, \ p_i\neq p_j\text{ for }i\neq j\}.
$$
For each $q \in Q$, we will construct a non-canonical element $\xi_q \in R_q/D_q$, all of which will play the role of generators of $B_d / (D_d + S_d)$.

\begin{defi}
For each $p \mid q$ with $p \in P$, fix a primitive root $\mu_p$ modulo $p$. To define $\xi_q$, it suffices to define its $d$-parts for $d\mid q$. We define its primitive part as
\begin{eqnarray*}
\xi_{q, 1} &:= & 
\prod_{\substack{p \in P \\ p \mid q}} (1, \mu_p, ..., \mu_p^{\varphi(p)/2 - 1})_p,
\end{eqnarray*}
where $\varphi$ is Euler's totient function. This is 
$\prod_{\substack{3\leq p \hbox{ prime  }   \\ p \mid q}} (1, \mu_p, ..., \mu_p^{\frac{p-1}{2} - 1})_p$  if $4\nmid q$ and and $(1)_4$ times the same expression if $p\mid q$.  

For $d \mid q$ with $1 < d < q$, we define the $d$-part $\xi_{q, d}$ inductively as the solution to the equation
\[\varphi(d) \xi_{q, d} = \sum_{\substack{e \mid d \\ e < d}} \tau_d(\xi_{q, e}^{(q)}) \in R_{q/d} / D_{q/d},\]
where the map $\tau_d = \tau_{q, d} : R_q \rightarrow R_{q/d}$ is defined as
\begin{equation}
    \tau_{d}((a)_q) := \frac{\varphi(q)}{\varphi(q')} g_{q,d}\cdot(a')_{q/d}, \ \   \text{ where }\ \ g_{q,d}:= \prod_{\substack{p \text{ prime} \\ p \mid d' \\ p \nmid q/d}} (1, -p^{-1})_{q/d}\in\Z[G_{q/d}], 
\end{equation}
and 
$$q' = q / \gcd(q,a),\ \  d' = d/\gcd(q,a),\ \  a' = a/\gcd(q,a).
$$
Since we are modding out by $D_{q/d}$, such a solution always exists. Then, $\xi_q = \sum_{d\mid q} \xi_{q, d}$. We call $\{\xi_q \;:\; q \in Q\}$ the set of exceptional elements.
\end{defi}

\begin{prop}[\cite{Aoki1983}] $\xi_q \in B_q / D_q$.
\end{prop}

\begin{defi}
For every $d\in\N$, we define the module of exceptional tuples $T_d\subseteq R_d/D_d$ as the $G_d$-submodule generated by the elements $\xi_q$ for $q \in Q$ with $q \mid d$.
\end{defi} 

\begin{theo}[\cite{Aoki1983}]\label{thm:generators_of_Bm} The $G_d$-module $B_d/(D_d + S_d)$ is generated by $T_d$.
\end{theo}

\begin{coro}\label{cor:reduce_hodge_conjecture_to_Tm} If $T_d \subseteq C_d/D_d$, then the Hodge conjecture is true for all weighted Fermat varieties of degree $d$.
\end{coro}
\begin{proof} By \cref{thm:generators_of_Bm}, $B_d/D_d = T_d+((D_d + S_d)/D_d)$. By \eqref{eqlinstdalg} and the hypothesis, $B_d \subseteq C_d$, so the Hodge conjecture holds by \cref{thm:reduce_hodge_conjecture_to_Bm}.
\end{proof}

\begin{coro}[\cite{aoki1987}]\label{correddivQ}
For every natural $d\in \N$ of the form $d=2^{e_0}p_1^{e_1}\cdots p_k^{e_k}$ with $p_i\in P$ different, $e_i>0$ for $i=1,\ldots,k$ and $e_0\in\{0,1\}$, the Hodge conjecture for Fermat varieties holds in degree $d$, if and only if, it holds in degree $d_0:=p_1\cdots p_k$. Furthermore, if $k$ is odd, then it also holds, if and only if, it holds in degree $d_0/p_i$ for all $i=1,\ldots,k$.
\end{coro}

\begin{proof}
Just note that $T_d=T_{d_0}$ and for $k$ odd, $T_{d_0}$ is generated by $\bigcup_{i=1}^kT_{d_0/p_i}$.
\end{proof}

\begin{coro}
Let $p\in\N$ be a prime number, and $d=p^e$ or $d=2p^e$ for some $e\ge 1$, then $B_d=D_d+S_d$.
\end{coro}

\begin{proof}
In those cases there are no exceptional tuples.
\end{proof}

\begin{rem}
In particular, the Hodge conjecture holds for Fermat varieties of degree a prime power or twice a prime power.
\end{rem}

We have thus reduced the Hodge conjecture for Fermat varieties of degree $d$ to showing that $T_d \subseteq C_d/D_d$, which reduces to show that $\xi_q\in C_q/D_q$ for all $q\in Q$ such that $q\mid d$. Although this problem is still open, the following important result by Kubert (reproved by Aoki) is the best available approximation to a solution.

\begin{theo}[\cite{kubert1979universal, Aoki1983}]\label{thm:Tm_is_killed_by_two}
For every $q\in Q$
\begin{equation}
2\xi_q \in (D_q + S_q)/D_q.
\end{equation}
And so, the module $B/C$ is of $2$-torsion.
\end{theo}

We state this result here because later, in the fourthcoming article  
we will use it to compute algebraic values of Gamma products.

\subsection{Length of a tuple and Hodge conjecture}\label{sec:4.3}

All results about Hodge conjecture for Fermat varieties rely on the inductive structure of Fermat varieties (c.f. \cite[Theorem II]{sh79}), which in terms of characters corresponds to the join operation. More precisely, the general strategy is to write any $\alpha\in\mathfrak{B}^n_d$ as a join of indecomposable Hodge characters and then study the algebraicity of indecomposable characters. In order to apply an inductive argument on $n$, Shioda notes that it is enough to show that every indecomposable Hodge character $\alpha$ is quasi-decomposable which means that for some $\delta\in\mathfrak{D}^0_d$, the character $\alpha*\delta$ can be decomposed as a join of non-linear Hodge characters. Note that the induction argument starts from $n\ge 4$, since for $n=2$ all Hodge characters are known to be algebraic by Lefschetz $(1,1)$ theorem. With this strategy Shioda proved the Hodge conjecture for Fermat varieties of degree $\le 20$ by hand, and with computational assistance da Silva Jr. proves the degree $21$ case \cite{da2021notes}.  Our strategy to prove the Hodge conjecture from the inductive structure is similar, the main difference is that we incorporate the join, deletion and lift of all known algebraic characters to reduce the length of all exceptional tuples, which (by \cref{thm:generators_of_Bm}) are the only ones left to consider. We remark that the idea of incorporating join and deletion of standard characters was already used by Aoki in \cite[Theorem 1-3]{aoki1987} to prove the Hodge conjecture for Fermat varieties of degree a prime power. Therefore, the main novelty of our algorithm is to incorporate the lift operation, which actually improves all previous known methods. This justifies why we work with the module $R$ instead of just considering the $G_d$-module $R_d$.

\begin{defi}
For every $\alpha\in R$ of the form $\alpha=\sum_{i=1}^kn_i(a_i)_{m_i}$ we define its length as
$$
\ell(\alpha):=\sum_{i=1}^k|n_i|\in \Z_{\ge 0}.
$$
By taking quotients by different submodules, we introduce the following lengths
\begin{equation}
\ell_1(\alpha):=\min\{\ell(\alpha'): \alpha-\alpha'\in D\},
\end{equation}
\begin{equation}
\ell_2(\alpha):=\min\{\ell(\alpha'):\alpha-\alpha'\in D+\widetilde{S}\},
\end{equation}
\begin{equation}
\mathscr{L}(\alpha):=\min\{\ell(\alpha'): \alpha-\alpha'\in C\}.
\end{equation}
Furthermore, when $\alpha\in R_d$ we define
\begin{equation}
\ell_{2,d}(\alpha):=\min\{\ell(\alpha'):\alpha-\alpha'\in D_d+S_d\}.
\end{equation}
\end{defi}

\begin{rem}
It is clear from the definition that
\begin{equation}
\mathscr{L}(\alpha)\le \ell_2(\alpha)\le \ell_1(\alpha).
\end{equation}
Furthermore, for $d'\mid d$  and $\alpha\in R_{d'}\xhookrightarrow{(\cdot)^{(d)}}R_d$ we have
\begin{equation}
\ell_2(\alpha)\le \ell_{2,d}(\alpha)\le\ell_{2,d'}(\alpha).
\end{equation}
In fact, $\ell_2(\alpha)=\min\{\ell_{2,d}(\alpha): \alpha\in R_d\}$.
\end{rem}

\begin{rem}\label{remefftuples}
Since $(a_i)_{m_i}+(-a_i)_{m_i}\in D$, for effects of computing any of the above lengths we can always replace $n_i(a_i)_{m_i}$ by $-n_i(-a_i)_{m_i}$ and so we can assume that $n_i\ge 0$ for all $i=1,\ldots,k$. In other words, we can assume that $\alpha\in\mathfrak{A}^n_m$ for $n=\ell(\alpha)-2$.
\end{rem}

\begin{prop}\label{prop:lengthredto6}
Given $\alpha\in\mathfrak{B}^n_m$ such that $\mathscr{L}(\alpha)\le 6$, then $\alpha\in\mathfrak{C}^n_m$ and so $\mathscr{L}(\alpha)=0$.
\end{prop}

\begin{proof}
By definition there exists $\alpha'\in\mathfrak{B}^{\mathscr{L}(\alpha)-2}_{m'}$ and $\gamma\in C$ such that $\alpha-\alpha'=\gamma$. Since $\mathscr{L}(\alpha)-2\le 4$, it follows by Lefschetz $(1,1)$ theorem and \cite[Corollary 3.2]{kang2016refined} that $\mathfrak{B}^{\mathscr{L}(\alpha)-2}_m=\mathfrak{C}^{\mathscr{L}(\alpha)-2}_m$, and so $\alpha=\alpha'+\gamma\in C$.
\end{proof}

In this way, we have reduced the Hodge conjecture to show that $\mathscr{L}(\xi_q)=0$ for all $q\in Q$, for which is enough to bound $\mathscr{L}(\xi_q)\le 6$. In what follows we will present the results of two length reduction algorithms. These algorithms are explained in detail in the next subsection. The first alorithm, which we call Aoki-Shioda algorithm, reduces each $\xi_q\in B_q$ by the elements in $D_q$ and $S_q$ and so it gives upper bounds for $\ell_{2,q}$. The second algorithm incorporates the lift, and so reduces $\xi_q$ by the elements in $D_d$ and $S_d$ for some $q\mid d$, we call it the lifted Aoki-Shioda algorithm and gives upper bounds for $\ell_2$. For small lengths we present for each $\xi_q$ a representative $\alpha'\in B$ such that $\ell(\alpha')$ attains the bound resulting of running the algorithm. As a consequence of the results presented in \cref{tab:representatives_for_exceptional_cycles_AS} and \cref{tab:representatives_for_exceptional_cycles_liftedAS} we obtain the following result.

\begin{theo}\label{mainthm}
The Hodge conjecture holds for $X^n_d$ for $d< 65$ and $d\neq 44, 51, 52$.
\end{theo}

\begin{table}[H]
  \centering
  \begin{tabular}[t]{|c|c||c|c||c|c||c|c||c|c||c|c||c|c|}
    \hline
    $q$ & $\ell$ & $q$ & $\ell$ & $q$ & $\ell$ & $q$ & $\ell$ & $q$ & $\ell$ & $q$ & $\ell$ & $q$ & $\ell$\\
    \hline
    12 & 3 & 69 & 12 & 129 & 22 & 188 & 24 & 247 & 56 & 305 & 62 & 371 & 80 \\ \hline
    15 & 4 & 76 & 11 & 133 & 28 & 201 & 36 & 249 & 42 & 309 & 54 & 377 & 86 \\ \hline
    20 & 3 & 77 & 16 & 141 & 24 & 203 & 42 & 253 & 56 & 316 & 40 & 381 & 64 \\ \hline
    21 & 4 & 85 & 18 & 143 & 32 & 205 & 42 & 259 & 54 & 319 & 72 & 388 & 50 \\ \hline
    28 & 4 & 87 & 16 & 145 & 32 & 209 & 46 & 265 & 54 & 321 & 54 & 391 & 90 \\ \hline
    33 & 6 & 91 & 22 & 148 & 19 & 212 & 27 & 267 & 46 & 323 & 74 & 393 & 66 \\ \hline
    35 & 8 & 92 & 12 & 155 & 30 & 213 & 36 & 268 & 35 & 327 & 54 & 395 & 80 \\ \hline
    39 & 6 & 93 & 16 & 159 & 28 & 215 & 44 & 284 & 36 & 329 & 70 & 403 & 94 \\ \hline
    44 & 7 & 95 & 20 & 161 & 34 & 217 & 46 & 287 & 64 & 332 & 43 & 404 & 51 \\ \hline
    51 & 10 & 111 & 20 & 164 & 22 & 219 & 42 & 291 & 50 & 335 & 70 & 407 & 96 \\ \hline
    52 & 7 & 115 & 24 & 172 & 25 & 221 & 54 & 292 & 36 & 339 & 58 & 411 & 70 \\ \hline
    55 & 10 & 116 & 15 & 177 & 30 & 235 & 48 & 295 & 60 & 341 & 76 & 412 & 52 \\ \hline
    57 & 10 & 119 & 26 & 183 & 36 & 236 & 31 & 299 & 68 & 355 & 70 & 413 & 88 \\ \hline
    65 & 16 & 123 & 26 & 185 & 38 & 237 & 40 & 301 & 70 & 356 & 44 & 415 & 84 \\ \hline
    68 & 10 & 124 & 16 & 187 & 42 & 244 & 31 & 303 & 52 & 365 & 74 & 417 & 70 \\ \hline
  \end{tabular}
  \caption{Length of representatives for $\xi_q$ via the Aoki-Shioda algorithm.}
\label{tab:lengths_of_representatives_for_exceptional_cycles_AS}
\end{table}

\begin{table}[H]
    \centering
    \renewcommand{\arraystretch}{1.2}
    \begin{tabular}{c | l | l }
        $q$ & representative of $\xi_q$ & upper bound for $\ell_{2,q}(\xi_q)$ \\ \hline
        12 & $(1, 8, 9)_{12}$ & 3 \\ 
  15 & $(1, 7, 10, 12)_{15}$ & 4 \\ 
  20 & $(3, 8, 19)_{20}$ & 3 \\ 
  21 & $(1, 10, 15, 16)_{21}$ & 4 \\ 
  28 & $(5, 13, 17, 21)_{28}$ & 4 \\ 
  33 & $(1, 4, 16, 22, 25, 31)_{33}$ & 6 \\
  35 & $(1, 2, 16, 17, 21, 22, 30, 31)_{35}$ & 8 \\
  39 & $(1, 16, 19, 22, 28, 31)_{39}$ & 6 \\ 
  44 & $(8, 11, 15, 19, 27, 31, 43)_{44}$ & 7 \\ 
    \end{tabular}
    \caption{Representatives for $\xi_q$ for small $q$, via the Aoki-Shioda algorithm.}
\label{tab:representatives_for_exceptional_cycles_AS}
\end{table}

\begin{table}[H]
  \centering
\begin{tabular}{|c|c|c||c|c|c||c|c|c||c|c|c|}
    \hline
    $q$ & $d$ & $\ell$ & $q$ & $d$ & $\ell$ & $q$ & $d$ & $\ell$ & $q$ & $d$ & $\ell$ \\
    \hline
    33 & 66 & 4 & 52 & 156 & 7 & 69 & 138 & 9 & 91 & 182 & 17 \\
    \hline
    35 & 70 & 6 & 55 & 110 & 5 & 76 & 228 & 11 & 92 & 276 & 12 \\
    \hline
    39 & 78 & 3 & 57 & 114 & 6 & 77 & 154 & 12 & 93 & 186 & 10 \\
    \hline
    44 & 132 & 7 & 65 & 130 & 10 & 85 & 170 & 12 & 95 & 190 & 10 \\
    \hline
    51 & 102 & 7 & 68 & 204 & 8 & 87 & 174 & 10 & 111 & 222 & 10 \\
    \hline
\end{tabular}
\caption{Length of representatives for $\xi_q$ via the lifted Aoki-Shioda algorithm.}
\label{tab:lengths_of_representatives_for_exceptional_cycles_liftedAS}
\end{table}

\begin{table}[H]
    \centering
    \renewcommand{\arraystretch}{1.2}
    \begin{tabular}{c | l | l }
        $q$ & representative of $\xi_q$ & upper bound for $\ell_2(\xi_q)$ \\ \hline
        33 & $(19, 26, 43, 44)_{66}$ & 4 \\
  35 & $(10, 12, 13, 53, 56, 66)_{70}$ & 6 \\
  39 & $(25, 43, 49)_{78}$ & 3 \\
  44 & $(3, 37, 39, 75, 84, 99, 125)_{132}$ & 7 \\
  51 & $(7, 19, 25, 60, 68, 87, 91)_{102}$ & 7 \\
  52 & $(24, 51, 63, 75, 87, 97, 149)_{156}$ & 7 \\
  55 & $(19, 29, 39, 79, 109)_{110}$ & 5 \\ 
  57 & $(23, 41, 64, 66, 71, 77)_{114}$ & 6 \\
    \end{tabular}
    \caption{Representatives for $\xi_q$ for small $q$, via the lifted Aoki-Shioda algorithm.}
\label{tab:representatives_for_exceptional_cycles_liftedAS}
\end{table}

\begin{rem}
In this way we see that the upper bounds for the lengths reduces the Hodge conjecture of a certain degree to a search in a fixed dimension. On the other hand, to obtain lower bounds for the lengths is a quite hard problem and is related to determine relations between de the degree and the dimension where we can describe precisely all Hodge characters. For instance as a corollary of \cref{theoAokilinearcycles} we obtain the following lower bounds.
\end{rem}

\begin{prop}
Let $\alpha\in B_d$ and $p\mid d$ its smallest prime divisor.
\begin{itemize}
    \item[(i)] If $\alpha\notin D_d$, then $\ell_1(\alpha)\ge p$.
    \item[(ii)] If $\alpha\notin D_d+S_d$, then $\ell_{2,d}(\alpha)\ge p$.
\end{itemize}
\end{prop}

\begin{proof}
(i) If $\ell_1(\alpha)=\ell(\alpha-\alpha')=n+2\le p-1$ for some $\alpha'\in D$, then $\text{gcd}(d,(n+2)!)=1$ and so by \cref{theoAokilinearcycles} it follows that $\alpha-\alpha'\in D$, hence $\alpha\in D$.

\noindent(ii) If $\ell_{2,d}(\alpha)=\ell(\alpha-\alpha')=n+2\le p-1$ for some $\alpha'\in D_d+S_d$, then $\text{gcd}(d,(n+2)!)=1$ and so by \cref{theoAokilinearcycles} we have that $\alpha-\alpha'\in D_d$, hence $\alpha\in D_d+S_d$.
\end{proof}

\begin{rem}
An immediate consequence of the previous proposition is that if some $\alpha\in B_d$ is such that $\ell_2(\alpha)<p$ for $p\mid d$ its smallest prime divisor, then there exists some $d\mid d'$ such that $\ell_2(\alpha)=\ell_{2,d'}(\alpha)$ and so if $p'\mid d'$ is its smallest prime divisor, then $p'<p$. In other words, in order to reduce the length $\ell_{2,d}(\alpha)$ we have to lift to a multiple of $d$ which is divisible by a small prime $p'<p$.
\end{rem}

\begin{conj}\label{conj1}
There exists an unbounded function $f:\N\to\N$ with finite fibers such that for all $\alpha\in B_d\setminus(D_d+S_d)$
$$
\ell_{2,d}(\alpha)\ge f(d).
$$
\end{conj}

\begin{rem}
The previous conjecture is equivalent to say that the function
$$
f_0(d):=\min\{\ell_{2,d}(\alpha):\alpha\in B_d\setminus(D_d+S_d)\}
$$
has finite fibers.
\end{rem}

\begin{coro}
If \cref{conj1} holds, the Hodge conjecture for Fermat varieties of dimension $n$ reduces to the Hodge conjecture for Fermat varieties of dimension $n$ and degree $d$ such that $f_0(d)\le n+2$. In particular, it is reduced to the finite collection of exceptional tuples $\xi_q$ with $q\mid d$ such that $f_0(d)\le n+2$.
\end{coro}

\begin{rem}
\cref{tab:lengths_of_representatives_for_exceptional_cycles_liftedAS} suggests that the Hodge conjecture for Fermat sixfolds holds, if and only if, the tuples $\xi_{44}$, $\xi_{51}$, $\xi_{52}$ and $\xi_{68}$ are algebraic. Using the same table one can do a similar (but less reliable) guess for eightfolds. For the case of surfaces, by the Table in \cite{Shioda1981}, we see that the exceptional tuples $\xi_q$ with some multiple $q\mid d$ such that $f_0(d)\le 4$ are precisely $q<44$ with $q\neq 35$. Therefore, the smallest degree Fermat variety admitting a exceptional cycle not coming from surfaces is $X^6_{35}$, while the smallest dimension one admitting it is $X^4_{70}$. Our final remark is that, by \cref{tab:lengths_of_representatives_for_exceptional_cycles_liftedAS}, the smallest degree and dimension Fermat variety where the Hodge conjecture remains open is $X^6_{44}$.
\end{rem}

\subsection{Length reduction algorithms}

In this section we present the length reduction algorithms used to obtain the results of the tables shown in the previous section. Let us start explaining how the Aoki-Shioda algorithm works.

Fix a degree $d$. We work with the module $R_d$, representing each term $$\alpha=\sum_{i=1}^{d-1}n_i (i)_d\in R_d$$ by the vector $v=(n_1,\ldots,n_{d-1})\in\Z^{d-1}$. Let $A\in(\mathbb{Z}_{\geq0})^{(d-1)\times k}$ be the matrix whose columns are an explicit finite set of generators of $D_d+S_d$. Concretely, we take the $G_d$-orbit of $(1,-1)$, together with $\varepsilon_d$ if $d$ is even, and the standard characters $s_{p,i}$ for each prime $p\mid d$. Since $R_d$ has rank $d-1$, the matrix $A$ has exactly $d-1$ rows, and we write $k$ for its number of columns. For example, for $d=6$, a matrix of generators $D_6 + S_6$ is
\begin{align*}
    \begin{pmatrix}
        1 & 0 & 0 & 1 & 0 \\
        0 & 1 & 0 & 0 & 2 \\
        0 & 0 & 1 & 1 & 1 \\
        0 & 1 & 0 & 2 & 0 \\
        1 & 0 & 0 & 0 & 1
    \end{pmatrix}.
\end{align*}

Given a character $\alpha\in \mathfrak{B}^n_d$, it has an associated vector $v\in(\mathbb{Z}_{\geq0})^{d-1}$ under the identification of \cref{remchartotuples}, obtained by recording, for each $a\in(\mathbb{Z}/d\mathbb{Z})\setminus\{0\}$, the number of times $a$ occurs among the coordinates of $\alpha$. So, in order to calculate $\ell_{2,d}(\alpha)$ we need to find $x\in\mathbb{Z}^k$ such that $v-Ax$ has the smallest possible norm. In other words, we want to minimize
\[\|v-Ax\|_1 = \sum_{i=1}^{d-1}\lvert n_i - \sum_{j=1}^k A_{ij}x_j \rvert,\]
where $\lvert\cdot\rvert$ denotes the absolute value. By \cref{remefftuples} we may assume, without loss of generality, that the coefficients of the reduced tuple $v-Ax$ are all non-negative. Under this assumption the $1$-norm becomes simply the sum of the coordinates of the vector, so the optimization problem above can be rewritten as
\begin{equation}
\label{eq:linprogprob}    
    \min_{\begin{smallmatrix}
        x\in\mathbb{Z}^k\\ \sum_{j=1}^kA_{ij}x_j\le n_i\\ \forall j=1,\ldots,d-1
    \end{smallmatrix}} \sum_{i=1}^{d-1}\left( n_i - \sum_{j=1}^k A_{ij}x_j\right)= \max_{\begin{smallmatrix}
        x\in\mathbb{Z}^k\\ \sum_{j=1}^kA_{ij}x_j\le n_i\\ \forall j=1,\ldots,d-1
    \end{smallmatrix}} \sum_{j=1}^kc_jx_j\ , \ \ \text{ where }\ c_j:=\sum_{i=1}^{d-1}A_{ij}.
\end{equation}
Which is an integral linear programming problem. Solving this problem for $d=q$ and $\alpha=\xi_q$ recovers values of $\ell_{2,q}(\xi_q)$. Due to our computational limitations we solved the above linear programming problem restricted to $x\in\{-1,0,1\}^k$ with the \texttt{milp} function of the SciPy library for Python which calls the HiGHS solver. This is the Aoki-Shioda algorithm. This finiteness search restriction imposed on $x$ is the reason why we only obtain the upper bounds for $\ell_{2,q}(\xi_q)$ presented in \cref{tab:representatives_for_exceptional_cycles_AS}.

The lifted Aoki-Shioda algorithm results by iterating the Aoki-Shioda algorithm over the lifts $\alpha^{(d)}$ for $d=\text{lcm}(e,q)$, $2\le e<q$, $q\nmid e$. This algorithm recovers the upper bounds for $\ell_2(\xi_q)$ presented in \cref{tab:representatives_for_exceptional_cycles_liftedAS}.

We present the pseudo-code of the Aoki-Shioda algorithm as follows: The function \texttt{as\_vector} (resp. \texttt{as\_tuple}) is the function that takes a tuple (resp. vector) and returns the vector (resp. tuple) associated to the tuple (resp. vector). In the following GitHub repository \href{https://github.com/mmiirandaaaa/Length-reduction-functions}{\texttt{https://github.com/mmiirandaaaa/Length-reduction-functions}} you can find the code executed to obtain the length reduction of the exceptional tuples of \cref{tab:lengths_of_representatives_for_exceptional_cycles_AS} and \cref{tab:lengths_of_representatives_for_exceptional_cycles_liftedAS}.

\begin{algorithm}
    \caption{\\
    \textbf{Input:} The set of exceptional tuples $T$. \\
    \textbf{Output:} A set $E$ of reduced exceptional tuples via the Aoki-Shioda Algorithm.
    }
    \begin{algorithmic}
        \State $E := \emptyset$ \Comment{In this set $\alpha = a\cdot\alpha$, where $a\in\mathbb{Z}$ and $\alpha$ is a tuple}
        \State $P := \{p\in\mathbb{N}\mid p \text{ prime or }p=4\}\setminus\{2\}$
        \State $Q := \prod_{p_1p_2\in P} p_1p_2 = \{12,15,20,21,\dots\}$
        \For{$q \in Q$}
            \State $L := \{lcm(n,q)\mid n\in\{1,\dots,q-1\}\}$
            \For{$l\in L$}
                \State Let $a_1,\dots,a_k\in\mathbb{Z}^{l-1}$ be the generators of $D_l+S_l$ written as column vectors
                \State $A := \begin{bmatrix} a_1 & a_2 & \dots & a_k\end{bmatrix} \in\mathbb{Z}^{(l-1)\times k}$
                \For{$e \in E$}
                    \If{\texttt{e.deg} $\mid l$} \Comment{\texttt{e.deg} is the degree where the exceptional $e$ appears the first time}
                        \State $e \gets \frac{l}{\texttt{e.deg}} \cdot e$
                        \State $e \gets \texttt{as\_vector}(e)$
                        \State $A \gets \begin{bmatrix}A \mid e\end{bmatrix}$
                    \EndIf
                \EndFor
                \State $t := \texttt{as\_vector}(\xi_q)$
                \State $x\gets$ Solution of the integer linear program given by $\min_{x\in\mathbb{Z}^{col(A)}}{\|t-Ax\|_1}$ subject to $Ax\leq t$
                \State $x \gets \texttt{as\_tuple}(x)$
                \If{\texttt{length}(x) $\leq 4$}
                    \State $E \gets E\cup\{x\}$
                \EndIf
            \EndFor
        \EndFor
        \State \Return $E$
    \end{algorithmic}
\end{algorithm}

\subsection{Fermat surfaces revisited}

In this final section we revisit the main result of \cite{Shioda1981} encoded in a table where he describes all the degrees $d$ where \textit{exceptional} Hodge characters appear in the Fermat surface $X^2_d$. In that work the word exceptional is used to essentially mean characters which are not in $\mathfrak{D}^2_d\cup\mathfrak{S}^2_d$, nor of the form 
$$
s_{2,i}+s_{2,2i}-(2i,d-2i)=\left(i,\frac{d}{2}+i,\frac{d}{2}+2i,d-4i\right),
$$
and nor lifts of characters of $\mathfrak{B}^2_{d'}$ for some $d'$ a proper divisor of $d$. Since by the time the decomposition theorem of Aoki (\cref{thm:generators_of_Bm}) was not available, Shioda measured this difference with a number $\Delta(d)$. Using the Aoki-Shioda algorithm we revisited this table by decomposing all Hodge characters for all such degrees in terms of linear, standard and exceptional characters (in the sense of \cref{sec:4.2}) and count how many characters have a similar type decomposition. We collect these foundings in the following table. In the second column we describe the number of characters of a given decomposition type. For instance, the equality $D+s_2+s_3+\xi_{12}=48$ in degree $12$ means that among the $642$ Hodge characters in $X^2_{12}$, exactly $48$ were written as a non-trivial linear combination of linear characters, standard characters associated to the primes $2$ and $3$, the exceptional character $\xi_{12}$ and their permutations. 

\begingroup
\setlength{\LTleft}{0pt}\setlength{\LTright}{0pt}
\setlength{\LTcapwidth}{\textwidth}
\small
\ifdefined\degcolwd\else\newlength{\degcolwd}\newlength{\degcolaux}\fi
\settowidth{\degcolwd}{$180=2^{2}\cdot 3^{2}\cdot 5$}
\settowidth{\degcolaux}{$156=2^{2}\cdot 3\cdot 13$}\ifdim\degcolaux>\degcolwd\setlength{\degcolwd}{\degcolaux}\fi
\settowidth{\degcolaux}{$120=2^{3}\cdot 3\cdot 5$}\ifdim\degcolaux>\degcolwd\setlength{\degcolwd}{\degcolaux}\fi
\begin{longtable}{@{}>{$}l<{$}@{\hspace{1.2em}}>{\raggedright\arraybackslash}p{\dimexpr\textwidth-\degcolwd-1.2em\relax}@{}}
\caption{Decomposition of the Hodge characters of the Fermat surfaces $X^2_d$.}
\label{tab:hodge-characters-surfaces}\\
\toprule
\multicolumn{1}{@{}l@{\hspace{1.2em}}}{$d$} & \textbf{Hodge characters} \\
\midrule
\endfirsthead
\caption[]{(continued)}\\
\toprule
\multicolumn{1}{@{}l@{\hspace{1.2em}}}{$d$} & \textbf{Hodge characters} \\
\midrule
\endhead
\multicolumn{2}{r@{}}{\footnotesize\itshape continues on the next page}\\
\endfoot
\bottomrule
\endlastfoot
2 & $D=1$ \\
\midrule
3 & $D=6$ \\
\midrule
4=2^{2} & $D=19$ \\
\midrule
6=2\cdot 3 & $D=61$, $s_{2}=24$ \\
\midrule
8=2^{3} & $D=127$, $s_{2}=48$ \\
\midrule
10=2\cdot 5 & $D=217$, $D+s_{2}=48$, $s_{2}=96$ \\
\midrule
12=2^{2}\cdot 3 & $D=331$, $D+\xi_{12}=24$, $D+s_{2}=24$, $D+s_{2}+\xi_{12}=72$, $D+s_{2}+s_{3}+\xi_{12}=48$, $\xi_{12}=24$, $s_{2}+s_{3}+\xi_{12}=24$, $s_{2}=72$, $s_{3}=24$ \\
\midrule
14=2\cdot 7 & $D=469$, $D+s_{2}=192$, $s_{2}=144$ \\
\midrule
15=3\cdot 5 & $D=474$, $D+\xi_{15}=24$, $D+s_{3}+\xi_{15}=48$, $D+s_{3}+s_{5}=72$, $D+s_{3}+s_{5}+\xi_{15}=96$, $\xi_{15}=24$, $s_{3}=96$ \\
\midrule
18=2\cdot 3^{2} & $D=817$, $D+s_{2}=96$, $D+s_{2}+s_{3}=480$, $s_{2}+s_{3}=24$, $s_{2}=144$, $s_{3}=96$ \\
\midrule
20=2^{2}\cdot 5 & $D=1027$, $D+s_{2}=36$, $D+s_{2}+\xi_{20}=576$, $D+s_{2}+s_{5}=156$, $s_{2}=192$ \\
\midrule
21=3\cdot 7 & $D=1044$, $D+\xi_{21}=24$, $D+s_{3}=48$, $D+s_{3}+\xi_{21}=96$, $D+s_{3}+s_{7}=48$, $D+s_{3}+s_{7}+\xi_{21}=120$, $D+s_{7}+\xi_{21}=24$, $\xi_{21}=24$, $s_{3}=144$ \\
\midrule
24=2^{3}\cdot 3 & $D=1519$, $D+\xi_{12}=24$, $D+s_{2}=24$, $D+s_{2}+\xi_{12}=24$, $D+s_{2}+s_{3}=432$, $D+s_{2}+s_{3}+\xi_{12}=600$, $\xi_{12}=24$, $s_{2}+\xi_{12}=24$, $s_{2}+s_{3}=24$, $s_{2}+s_{3}+\xi_{12}=72$, $s_{2}=192$, $s_{3}=120$ \\
\midrule
28=2^{2}\cdot 7 & $D=2107$, $D+s_{2}=480$, $D+s_{2}+\xi_{28}=48$, $D+s_{2}+s_{7}+\xi_{28}=48$, $s_{2}=288$ \\
\midrule
30=2\cdot 3\cdot 5 & $D=2437$, $D+\xi_{15}=24$, $D+s_{2}+s_{3}+s_{5}=1632$, $D+s_{2}+s_{3}+s_{5}+\xi_{15}=888$, $\xi_{15}=24$, $s_{2}+s_{3}=24$, $s_{2}+s_{3}+s_{5}=168$, $s_{2}+s_{3}+s_{5}+\xi_{15}=24$, $s_{2}=216$, $s_{3}=192$ \\
\midrule
36=2^{2}\cdot 3^{2} & $D=3571$, $D+\xi_{12}=24$, $D+s_{2}+s_{3}=744$, $D+s_{2}+s_{3}+\xi_{12}=408$, $\xi_{12}=24$, $s_{2}+s_{3}=216$, $s_{2}+s_{3}+\xi_{12}=24$, $s_{2}=288$, $s_{3}=216$ \\
\midrule
40=2^{3}\cdot 5 & $D=4447$, $D+s_{2}=36$, $D+s_{2}+\xi_{20}=360$, $D+s_{2}+s_{5}=348$, $D+s_{2}+s_{5}+\xi_{20}=600$, $s_{2}+s_{5}=24$, $s_{2}=408$ \\
\midrule
42=2\cdot 3\cdot 7 & $D=4921$, $D+s_{2}+s_{3}=96$, $D+s_{2}+s_{3}+s_{7}=1776$, $D+s_{2}+s_{3}+s_{7}+\xi_{21}=2784$, $D+\xi_{21}=24$, $s_{2}+s_{3}=96$, $s_{2}+s_{3}+\xi_{21}=48$, $s_{2}+s_{3}+s_{7}=480$, $s_{2}=48$, $\xi_{21}=24$, $s_{3}=120$ \\
\midrule
48=2^{4}\cdot 3 & $D=6487$, $D+\xi_{12}=24$, $D+s_{2}=36$, $D+s_{2}+\xi_{12}=24$, $D+s_{2}+s_{3}=900$, $D+s_{2}+s_{3}+\xi_{12}=744$, $\xi_{12}=24$, $s_{2}+\xi_{12}=24$, $s_{2}+s_{3}=120$, $s_{2}+s_{3}+\xi_{12}=120$, $s_{2}=384$, $s_{3}=312$ \\
\midrule
60=2^{2}\cdot 3\cdot 5 & $D=10267$, $D+\xi_{12}=24$, $D+\xi_{15}=24$, $D+s_{2}+s_{3}=24$, $D+s_{2}+s_{3}+\xi_{15}=24$, $D+s_{2}+s_{3}+s_{5}=3312$, $D+s_{2}+s_{3}+s_{5}+\xi_{12}+\xi_{15}=504$, $D+s_{2}+s_{3}+s_{5}+\xi_{12}+\xi_{15}+\xi_{20}=384$, $D+s_{2}+s_{3}+s_{5}+\xi_{15}=1008$, $D+s_{2}+s_{3}+s_{5}+\xi_{15}+\xi_{20}=2928$, $D+s_{2}+s_{3}+s_{5}+\xi_{20}=144$, $D+s_{3}+\xi_{15}=24$, $\xi_{12}=24$, $\xi_{15}=24$, $s_{2}+s_{3}=48$, $s_{2}+s_{3}+s_{5}=444$, $s_{2}+s_{3}+s_{5}+\xi_{12}=24$, $s_{2}+s_{3}+s_{5}+\xi_{15}=48$, $s_{2}=384$, $s_{3}+s_{5}=48$, $s_{3}=324$ \\
\midrule
66=2\cdot 3\cdot 11 & $D=12481$, $D+s_{2}+s_{3}=960$, $D+s_{2}+s_{3}+s_{11}+\xi_{33}=480$, $s_{2}+s_{3}=216$, $s_{2}=528$, $s_{3}=480$ \\
\midrule
72=2^{3}\cdot 3^{2} & $D=14911$, $D+\xi_{12}=24$, $D+s_{2}=48$, $D+s_{2}+s_{3}=1416$, $D+s_{2}+s_{3}+\xi_{12}=1200$, $\xi_{12}=24$, $s_{2}+s_{3}=360$, $s_{2}+s_{3}+\xi_{12}=96$, $s_{2}=576$, $s_{3}=504$ \\
\midrule
78=2\cdot 3\cdot 13 & $D=17557$, $D+s_{2}+s_{3}+s_{13}=864$, $D+s_{2}+s_{3}+s_{13}+\xi_{39}=768$, $s_{2}+s_{3}=48$, $s_{2}+s_{3}+s_{13}=996$, $s_{2}=300$, $s_{3}=120$ \\
\midrule
84=2^{2}\cdot 3\cdot 7 & $D=20419$, $D+\xi_{12}=24$, $D+s_{2}+s_{3}=192$, $D+s_{2}+s_{3}+s_{7}=3000$, $D+s_{2}+s_{3}+s_{7}+\xi_{12}=24$, $D+s_{2}+s_{3}+s_{7}+\xi_{12}+\xi_{21}+\xi_{28}=96$, $D+s_{2}+s_{3}+s_{7}+\xi_{21}+\xi_{28}=3648$, $D+s_{2}+s_{3}+s_{7}+\xi_{28}=48$, $D+\xi_{21}=24$, $\xi_{12}=24$, $s_{2}+s_{3}=120$, $s_{2}+s_{3}+\xi_{21}=24$, $s_{2}+s_{3}+s_{7}=1068$, $s_{2}+s_{3}+s_{7}+\xi_{12}=24$, $s_{2}+s_{3}+s_{7}+\xi_{21}=72$, $s_{2}=48$, $\xi_{21}=24$, $s_{3}+s_{7}=24$, $s_{3}=276$ \\
\midrule
90=2\cdot 3^{2}\cdot 5 & $D=23497$, $D+\xi_{15}=24$, $D+s_{2}+s_{3}=96$, $D+s_{2}+s_{3}+s_{5}=2688$, $D+s_{2}+s_{3}+s_{5}+\xi_{15}=1464$, $\xi_{15}=24$, $s_{2}+s_{3}+s_{5}=540$, $s_{2}+s_{3}+s_{5}+\xi_{15}=24$, $s_{2}=612$, $s_{3}+s_{5}=120$, $s_{3}=528$ \\
\midrule
120=2^{3}\cdot 3\cdot 5 & $D=42127$, $D+\xi_{12}=24$, $D+\xi_{15}=24$, $D+s_{2}+s_{3}=24$, $D+s_{2}+s_{3}+s_{5}=4320$, $D+s_{2}+s_{3}+s_{5}+\xi_{12}=24$, $D+s_{2}+s_{3}+s_{5}+\xi_{12}+\xi_{15}=1368$, $D+s_{2}+s_{3}+s_{5}+\xi_{12}+\xi_{15}+\xi_{20}=960$, $D+s_{2}+s_{3}+s_{5}+\xi_{15}=1032$, $D+s_{2}+s_{3}+s_{5}+\xi_{15}+\xi_{20}=3744$, $D+s_{2}+s_{3}+s_{5}+\xi_{20}=480$, $D+s_{3}+\xi_{15}=24$, $\xi_{12}=24$, $\xi_{15}=24$, $s_{2}+s_{3}+s_{5}=2340$, $s_{2}+s_{3}+s_{5}+\xi_{12}=96$, $s_{2}+s_{3}+s_{5}+\xi_{15}=48$, $s_{2}=12$, $s_{3}+s_{5}=48$, $s_{3}=48$ \\
\midrule
156=2^{2}\cdot 3\cdot 13 & $D=71611$, $D+\xi_{12}=24$, $D+s_{2}=12$, $D+s_{2}+s_{3}+s_{13}=1740$, $D+s_{2}+s_{3}+s_{13}+\xi_{12}=72$, $D+s_{2}+s_{3}+s_{13}+\xi_{12}+\xi_{39}=1200$, $D+s_{2}+s_{3}+s_{13}+\xi_{39}=192$, $\xi_{12}=24$, $s_{2}+s_{3}=120$, $s_{2}+s_{3}+s_{13}=708$, $s_{2}+s_{3}+s_{13}+\xi_{12}=24$, $s_{2}=1128$, $s_{3}+s_{13}=48$, $s_{3}=972$ \\
\midrule
180=2^{2}\cdot 3^{2}\cdot 5 & $D=95587$, $D+\xi_{12}=24$, $D+\xi_{15}=24$, $D+s_{2}+s_{3}+s_{5}=5208$, $D+s_{2}+s_{3}+s_{5}+\xi_{12}+\xi_{15}=600$, $D+s_{2}+s_{3}+s_{5}+\xi_{12}+\xi_{15}+\xi_{20}=1320$, $D+s_{2}+s_{3}+s_{5}+\xi_{12}+\xi_{20}=24$, $D+s_{2}+s_{3}+s_{5}+\xi_{15}=1248$, $D+s_{2}+s_{3}+s_{5}+\xi_{15}+\xi_{20}=2856$, $D+s_{2}+s_{3}+s_{5}+\xi_{20}=408$, $\xi_{12}=24$, $\xi_{15}=24$, $s_{2}+s_{3}+s_{5}=3204$, $s_{2}+s_{3}+s_{5}+\xi_{12}=24$, $s_{2}+s_{3}+s_{5}+\xi_{15}=48$, $s_{2}=372$, $s_{3}+s_{5}=48$, $s_{3}=168$ \\
\end{longtable}
\endgroup

Certainly, since the generators of $B_d$ are not independent in general, such decompositions may change depending on which solutions of the linear programming problem \eqref{eq:linprogprob} we find. But an interesting fact about the above table is that the only exceptional tuples appearing in the decomposition of Hodge characters of length $4$ are $\xi_q$ with $q<44$ and $q\neq 35$ (note also that these are the only $q\in Q$ dividing some $d$ in the table). This gives a stronger evidence that for $q=35$ or $q\ge 44$ we should have $\ell_{2,q}(\xi_q)>4$.

\begin{rem}
The reason why Shioda's table includes some cases with no exceptional characters (like $d\le 10$), is because he considers all cases where $\Delta(d)\neq 0$, not caring about the cases where the formula of $\Delta(d)$ is subtracting repeated terms more than once (this is why $\Delta(d)<0$ for $d\le 10$). Since this over-subtracting issue only happens for small degrees, it does not affect the main result of \cite{Shioda1981} which was latter proved by Aoki \cite[Theorem C]{Aoki1983}.
\end{rem}



